\documentclass[12pt]{amsart}
\usepackage{todonotes}
\usepackage{amsmath,mathrsfs}
\usepackage{amsthm}
\usepackage{amssymb}
\usepackage{enumerate}
\usepackage{xcolor}
\usepackage{comment}
\usepackage{tikz}
\usepackage{tikz-cd}
\usepackage{color} 
\definecolor{rossred}{rgb}{1.0,0.25,0.66}  
\definecolor{rossgreen}{rgb}{0.25,0.66,0.25} 
\definecolor{rossblue}{rgb}{0.25,0.66,1.0}
\definecolor{sashapurple}{rgb}{0.5,0.15,0.5}
\usepackage[pdfborder={0 0 0}]{hyperref}
\hypersetup{
    colorlinks=true,
    linkcolor=rossblue,
    citecolor=rossgreen
}

\numberwithin{equation}{section}
\theoremstyle{plain}
\newtheorem{theorem}{Theorem}[section]
\newtheorem{proposition}[theorem]{Proposition}
\newtheorem{lemma}[theorem]{Lemma}

\newtheorem{problem}[theorem]{Problem}
\newtheorem{corollary}[theorem]{Corollary}
\newtheorem{conjecture}[theorem]{Conjecture}

\newtheorem{question}[theorem]{Question}
\theoremstyle{definition}

\newtheorem{definition}[theorem]{Definition}

\newtheorem{remark}[theorem]{Remark}

\newtheorem{example}[theorem]{Example}
\newtheorem{notation}[theorem]{Notation}
\newtheorem{comments}[theorem]{Comments}

\AtEndEnvironment{example}{\popQED\endexample}

\usepackage{url}
\usepackage{enumitem}
\usepackage[utf8]{inputenc}
\usepackage[T1]{fontenc}
\usepackage[paperwidth = 8.5in, paperheight = 11in, inner = 1in, outer = 1in, top = 1in, bottom = 1in]{geometry}

\usepackage[dvipsnames]{xcolor}

\usepackage{tikz}
\usepackage{subcaption}
\usetikzlibrary{calc}

\usepackage{tikz}
\usetikzlibrary{arrows.meta}
\tikzset{node/.style={draw, very thick, circle, minimum size=1cm},
    arrow/.style={very thick, -Triangle}}

\newcommand{\trdeg}{\mbox{\rm trdeg}\,}

\newcommand{\htt}{\text{ht}}

\newcommand{\qf}{\text{QF}}
\newcommand{\rk}{\text{rank}}

\usepackage{mathtools}
\author[]{Stephen Landsittel}
\address[S. Landsittel]
{Institute of Mathematics, Hebrew University, Givat Ram, Jerusalem 91904, Israel}
\email{stephen.landsittel@mail.huji.ac.il}

\author[]{Eran Nevo}
\address[E. Nevo]
{Institute of Mathematics, Universidad de Valladolid,  47011 Valladolid, Spain. \and Einstein Institute of Mathematics, Hebrew University, Jerusalem~91904, Israel.}
\email{nevo@math.huji.ac.il}

\date{\today}

\begin{document}

\begin{abstract}
We give a systematic analysis of the analytic spread of the determinantal ideal $J_{G,H}$ arising from a pair of graphs $(G,H)$. We give sharp bounds for this analytic spread, and combinatorial conditions and obstructions for its maximality via a linear matroid. When $H$ is a single edge, $J_{G,H}$ is isomorphic to the binomial edge ideal $J_G$ and its 
analytic spread is shown to equal the rank of $G$ in Kalai's  $2$-hyperconnectivity matroid. 
\end{abstract}

\title{Analytic spread via 
linear matroids}

\maketitle


\section{Introduction}

The analytic spread of an ideal in a local ring was defined by Northcott and Rees~\cite{ReeseNortcotOG} in studying reductions of ideals. Analytic spread has been of wide interest in commutative algebra during the past three decades, for instance in studying algebraic properties of ideals such as their cores and residual intersections, see for example \cite{BPcores, COSTANTINI-cores, LJPcores-mon}. It is well-known that the analytic spread of an ideal in a local or standard-graded ring $R$ is at most the dimension of $R$, and this bound is sharp. Ideals of maximal analytic spread are of particular interest in the theory of multiplicities, see for instance \cite{mult-max_spread, spread2}. In \cite{HJ-jmult_mon} the analytic spread was studied using the theory of $j$-multiplicities, wherein they show that the analytic spread of an ideal is maximal if and only if the ideal has positive $j$-multiplicity. The relationship between analytic spread and the theory of integral closure was observed in \cite{Filt-spread2, CM-spread, Cut-spread_filt}, and analytic spread was studied in relation to asymptotic depth in \cite{MIRANDA_NETO_2017}. In the context of algebraic geometry, analytic spread describes some aspects of the birational geometry corresponding to blow-ups and their fibers, see \cite{Dale-AG, Hartshorne}. Some additional papers on analytic spread include~\cite{qbrings-spread, M-spread, Burch_1972, bivia2003analytic, spread1, spread2}.

The analytic spread of an ideal $I$ is, under mild assumptions, the minimal number of generators of a reduction of $I$, see e.g.~\cite[Theorem 1]{ReeseNortcotOG}, where $J$ is a reduction of  $I$ if $JI^n=I^{n+1}$ for some $n$.
Precisely, the analytic spread of an ideal $I$ in a local ring $(S,m)$ was defined in \cite{ReeseNortcotOG} to be one more than the degree of the eventual polynomial $P(n):=\dim_{S/m} I^n/mI^n$, or equivalently, the Krull dimension of the special fiber ring $\oplus_{n\geq 0}I^n/mI^n$ of $I$, and is denoted by $\ell(I)$. An analogous definition of analytic spread is adopted when $S$ is a standard-graded ring over a field, where $m$ is the graded maximal ideal. It was proven in \cite[Chapter 8]{SH} that if the field is perfect, then the analytic spread of $I$ is the minimal number of generators of $I$ up to integral closure.
We always have the analytic spread bounded between the height of the ideal and the Krull dimension of its hosting ring: $\htt (I)\leq \ell(I)\leq \dim (S)$. If $I$ is generated by forms of a common degree, then $\ell(I)\leq \mu(I)$, where $\mu(I)$ is the minimal number of generators of $I$. In \cite[Chapters 5 and 8]{SH} they give a detailed analysis of the general properties of analytic spread in the setting of local rings.

A combinatorial case where analytic spread has been computed is the binomial edge ideal $J_G$ associated to a finite graph $G$, defined to be the ideal in a polynomial ring over a field $F$ which is generated by the binomials $f_{i,j}:=x_iy_j-x_jy_i$ over all edges $\{i,j\}$ in $G$. See Remark \ref{rem:H=K_2}. Binomial edge ideals are exactly the ideals generated by  maximal minors of generic matrices with two rows and $|V(G)|$ columns: the subset of minors considered corresponds to $E(G)$.
Binomial edge ideals were introduced independently in \cite{HHHKR, ohtani2011graphs} and have been of recent interest, see e.g. \cite{Ene-Hertzog-Hibi_binom, binom-licci, EZ-binom_edge}. The analytic spread of binomial edge ideals for the class of closed graphs was studied in \cite{kumar2022rees}.

Our first observation gives a combinatorial characterization of $\ell(J_G)$, 
the analytic spread of $J_G$ over a field $F$ of characteristic zero or larger than $2^{|G|}$, which is described in the following theorem and corollary. For a positive integer $n$, let $\mathcal{H}^n_2$ be the 2-hyperconnectivity matroid introduced by Kalai in \cite{Kalai}, and let $\mathcal{H}_2(G)$ be the submatroid of $\mathcal{H}_2^n$ obtained from a graph $G$ on $n$ vertices, whose rank is the rank of $G$ in $\mathcal{H}_2^n$.
Using representability of algebraic matroids over fields of characteristic zero or sufficiently large, see \cite[Theorem 8]{BMS}, and the independence of the hyperconnectivity matroid of field characteristic, see \cite[Corollary 1.13]{tyomkyn-char_p_hyp}, we obtain the following theorem.

\begin{theorem}\label{thm-hyp}
    Over any field of characteristic zero or characteristic larger than $2^{n}$, for every graph $G$ on $n$ vertices, the algebraic matroid of the field extension $F\subset F(f_{i,j}\mid \{i,j\}\in E(G))$ is isomorphic to $\mathcal{H}_2(G)$. In particular, $\ell(J_G) = \rk (\mathcal{H}_2(G))\le 2n-3$.
\end{theorem}

We do not have a counterexample to Theorem \ref{thm-hyp} over nonzero characteristics smaller than $2^n$. We introduce the assumption on $\text{char}(F)$ as a technical condition for applying \cite[Theorem 8]{BMS}, see Question~\ref{q:char}.

Using Theorem \ref{thm-hyp} we recover an alternate proof of  \cite[Theorems 1.1 and 1.2]{LHVDRDOC} in the case that $F$ has characteristic zero or characteristic larger than $2^{n}$, which states that for a connected graph on $n$ vertices, $n-1\leq \ell(J_G)\leq 2n-3$ and $\ell(J_G) = |E(G)|$ for any unicycle graph $G$. More generally, Theorem \ref{thm-hyp} combined with~\cite[Theorems 4.2 and 4.4]{Bernstein} gives the following  combinatorial characterization of $\ell(J_G)$ over a field of characteristic zero or characteristic larger than $2^{|G|}$.

\begin{corollary}\label{cor-hyp}
Over any field of characteristic zero or characteristic larger than $2^{n}$,
    for every graph $G$ of order $n$, 

$\ell(J_G)$ equals the maximal size $|E(G')|$ running over all the subgraphs $G'$ of $G$ such that there exists an acyclic orientation of the edges of $G'$ that has no alternating closed trail.

In particular,     
$\ell(J_G) = |E(G)|$ if and only if there exists an acyclic orientation of $E(G)$ that has no alternating closed trail.
\end{corollary}

The proof of Theorem \ref{thm-hyp} combines an expression of $\ell(J_G)$ as the rank of the algebraic matroid of the field extension $F\subset F(f_{i,j}\mid \{i,j\}\in E(G))$, see Lemma \ref{lem0}, and the classical fact that the algebraic matroid is isomorphic to the linear matroid of its derivations, when the characteristic is zero or large enough; the latter matroid equals $\mathcal{H}_2(G)$.

Next, we consider the analytic spread of binomial edge ideals $J_{G,H}$ associated to a pair of finite graphs $G$ and $H$, which were introduced in \cite{Binom_gen}. The ideal $J_{G,H}$ is defined to be the ideal generated by the binomials $x_{i,k}x_{j,l}-x_{i,l}x_{j,k}$ over all edges $\{i,j\}\in E(G)$ and $\{k,l\}\in E(H)$. Binomial edge ideals of pairs of graphs are ideals generated by appropriate $2$-minors of generic matrices. In the case that $H = K_2$, we recover the binomial edge ideal of the graph $G$.

\begin{problem}\label{prob:char-l(J_{G,H})}
Find a combinatorial characterization of $\ell(J_{G,H})$.
\end{problem}

In \cite{Binom_gen_paths} the analytic spread of a pair of a path and a complete graph $J_{P_n,K_m}$ was calculated to be $nm - 3$ for every $n,m\geq 3$. However, outside of this case and the analytic spread of $J_G = J_{G,K_2}$, the analytic spread of $J_{G,H}$ is not well-understood. For a graph $G$, the ideals $J_{G,K_n}$ are called \emph{generalized binomial} edge ideals and were studied in \cite{GeneralizedBinom, GeneralizedBinomBlock}. 

Again, over a field of characteristic zero or sufficiently large characteristic, we describe a matrix $\mathcal{H}(G,H)$, whose linear matroid we compare to the analytic spread of $J_{G,H}$ using \cite[Theorem 8]{BMS}, to prove the following result in Section \ref{sec4}.
\begin{proposition}\label{propB}
    For any graphs $G$ and $H$, the rank of the Jacobian $\mathcal{H}(G,H)$ of the minimal binomial generators of $J_{G,H}$ equals the analytic spread of $J_{G,H}$, over any field of characteristic zero or larger than $2^{|V(G)||V(H)|}$.
\end{proposition}

In general, for ideals $I\subset J$ in a polynomial ring, we do not have $\ell(I)\leq \ell(J)$. However, this monotonicity property holds on pairs of graphs, as the above matrix $\mathcal{H}(G,H)$ shows:
\begin{corollary}\label{corr-mon} (Monotonicity)
    For every four graphs $G\subset G'$ and $H\subset H'$, we have $\ell(J_{G,H})\leq \ell(J_{G',H'})$ over any field of characteristic zero or larger than $2^{|V(G')||V(H')|}$.
\end{corollary}

The following theorem gives tight bounds for $\ell(J_{G,H})$.

\begin{theorem}\label{thmA}
    Let $m,n\geq 3$. Let $G$ and $H$ be graphs on $n$ and $m$ vertices (with no isolated vertices), resp. Then 
    the following statements hold over any field.
    \begin{enumerate}
        \item[(1)]Let $c_G$ and $c_H$ be the number of connected components in $G$ and $H$ resp. Then 
        \begin{equation*}
        (n-c_G)(m-c_H)\leq \ell(J_{G,H})\leq mn.
    \end{equation*}The lower bound is achieved for every pair of forests. The upper bound 
    is achieved for every pair of complete graphs.
    \item[(2)] When $G$ is the path on $n$ vertices and $H$ is a unicycle
    \begin{equation*}
        \ell(J_{G,H}) = (n-1)m \text{ }(= |E(G)||E(H)|).
    \end{equation*}
    \end{enumerate}
\end{theorem}
Over characteristic zero or larger than $2^{|V(G)||V(H)|}$, by counting rows, the matrix $\mathcal{H}(G,H)$ of Proposition \ref{propB} shows that $\ell(J_{G,H})\leq |E(G)||E(H)|$. 
As a special case of Problem~\ref{prob:char-l(J_{G,H})}, and in view of Proposition~\ref{propB}, we propose to characterize ``independent" pairs first:
\begin{problem}\label{prob:char_ind_(G,H)}
Characterize the pairs of graphs $(G,H)$ for which  $\ell(J_{G,H}) = |E(G)||E(H)|$. 
\end{problem}
Theorem \ref{thmA} gives some families of pairs of graphs which satisfy this equality. The following theorem, combined with the monotonicity Corollary~\ref{corr-mon} states combinatorial obstructions to the equality, namely specifies ``dependent" pairs.
\begin{theorem}\label{thmA-2}
    For connected graphs $G$ and $H$ on at least three vertices each, we have $\ell(J_{G,H})<|E(G)||E(H)|$ over any field of characteristic zero or larger than $2^{|V(G)||V(H)|}$, in each of the following cases:
    \begin{enumerate}
        \item[(1)] when $G$ contains a star of larger order than some cycle in $H$;
        \item[(2)] when both $G$ and $H$ contain a cycle, and at least one of the cycles is even;
        \item[(3)] when one of $G$ or $H$ contains a cycle, and the other contains a cycle with a leaf. 
    \end{enumerate}
\end{theorem}

\begin{example}
    One can find many examples of small pairs of graphs $(G,H)$ such that $\ell(J_{G,H})<|E(G)||E(H)|$, say over characteristic zero, by experimenting with the command analyticSpread in the computer algebra software \emph{Macaulay2}. For example if $H$ is the 
    diamond graph (See Figure~\ref{fig:small_graphs} on p.6), the fact that $\ell(J_{P_3,H})<|E(P_3)||E(H)|$ where $P_3$ is the path on $3$ vertices, combined with Corollary \ref{corr-mon} shows that $\ell(J_{G,H})<|E(G)||E(H)|$ for 
    any connected graph $G$ on at least three vertices.
\end{example}

Regarding bases, at least over characteristic zero or large enough, we conjecture:

\begin{conjecture}\label{conj:odd_cycles_basis}
Let $n,m\ge 3$. 
    A pair of graphs $(G,H)$ on $n$ and $m$ vertices resp. has the property that the rows of the Jacobian $\mathcal{H}(G,H)$ indexed by $E(G)\times E(H)$ form a basis to the row space of  $\mathcal{H}(K_n,K_m)$, equivalently $\ell(J_{G,H}) = |E(G)||E(H)| = |V(G)||V(H)|$, if and only if both $G$ and $H$ are disjoint unions of odd cycles.
\end{conjecture}

By Proposition \ref{propB} and \cite[Theorem 1.4]{spread2} we have that the rows of $\mathcal{H} (G,H)$ form a basis of the row space of $\mathcal{H}(K_n,K_m)$ if and only if $\varepsilon(J_{G,H})>0$, where $\varepsilon(I)$ is the epsilon multiplicity of an ideal $I$ in a polynomial ring, as defined in~\cite{UV} and is discussed in \cite{CL, Cut1}. Note that \cite{CL, Cut1} define epsilon multiplicity in a local ring, and we can compute the epsilon multiplicity of $J_{G,H}$ by localizing at the graded maximal ideal. 
However, epsilon multiplicity is often difficult to compute.
In Theorem \ref{thmA} and Theorem \ref{thmA-2} we give combinatorial conditions for independence and dependence resp. in the matroid $\mathcal{H}^{n,m}:=\mathcal{H}(K_n,K_m)$, bypassing this difficulty in these special cases.

The history of when algebraic independence can be represented by linear independence of derivations traces back to an 1841 paper of Jacobi \cite[p.340, Propositio 1]{Jacobi}, in which it was proven that the Jacobian of any real analytic functions (e.g. polynomials with real coefficients) vanishes if they are algebraically dependent. Jacobi's result precedes the modern languages of matroids, derivations, and field characteristic. Around a century later, Van der Waerden showed in \cite{VanDerWaerden} that algebraic independence forms a matroid. In 1957 Rado proved that if a matroid is representable over a field of characteristic zero (e.g. is an algebraic matroid over the rationals), then for $p>>0$ it is representable over some, and hence infinitely many, fields of characteristic $p$. See \cite[Theorem 6]{Rado}. Finally in \cite[Theorem 8]{BMS} it was shown that over any field of either characteristic zero, or sufficiently high characteristic, an algebraic matroid is linear. In the results of both \cite{BMS, Rado}, the characteristic size requirement depends on the particular algebraic matroid at hand. In \cite{ZS1}, Zariski and Samuel reformulated Jacobi's result as a correspondence between separating transcendence bases and derivations in the language of modern commutative algebra.

\begin{question}\label{q:char}
    Can the characteristic conditions in Theorem \ref{thm-hyp} and Proposition \ref{propB} be removed? Is it true that for any field $F$, every algebraic matroid $F\subset F(g_1,\ldots,g_N)$ defined by degree two forms $g_i$ is representable?
\end{question}

\textbf{Outline.}
In Section \ref{sec-prelim} we give preliminaries coming from commutative algebra, the theory of binomial edge ideals, and the theory of matroids, which will be used throughout the paper. In Section \ref{sec3} we prove Theorem \ref{thm-hyp} and Corollary \ref{cor-hyp}. In Section \ref{sec4}, we discuss how applying common graph operations, such as adding leaves and handles, to graphs $G$ and $H$, affects the analytic spread of $\ell(J_{G,H})$. In Section \ref{sec5} we prove Theorem \ref{thmA} (1). In Section \ref{sec6} we focus on independent pairs of graphs. Section \ref{sec-7:prfA_pt1} is dedicated to the proof of Theorem \ref{thmA} (2). In Sections \ref{sec-8:prfA_part2}, \ref{sec-9:pf2}, and \ref{sec10:pf-2_pt2} we prove parts (1), (2), and (3) of Theorem \ref{thmA-2}, respectively.

\section{Preliminaries}\label{sec-prelim}

In order to prove the main theorems, we utilize theorems and notation from graph theory, the theory of matroids, and commutative algebra. During this section we recall the background from these topics which will be essential to our analysis in Sections \ref{sec3} and onward.

\subsection{Graph Theory}

We recall some notation and definitions following 
\cite{diestel, bondy_murty-graph_theory}.

An (undirected) graph is a pair $G = (V,E)$ where $V$ is a set and $E$ is a family of doubletons $\{u,v\}$ with $u,v\in V$. We write $V(G) = V$ and $E(G) = E$. The elements of $V$ are called \emph{vertices} and the elements of $E$ are called \emph{edges}. $G$ is called \emph{simple} when $u\neq v$ for all $\{u,v\}\in E(G)$. When $|V(G)|$ is finite, we call $G$ a \emph{finite graph} and we call $|V(G)|$ the \emph{order} of $G$. If $v\in V(G)$ and there does not exist $e\in E(G)$ with $v\in e$, then we say that $v$ is \emph{isolated}. Isolated vertices in graphs $G$ and $H$ do not contribute to the ideal $J_{G,H}$. All graphs in this paper are assumed to be finite and simple with no isolated vertices, in particular $|E(G)|<\infty$. For convenience we write $|G|: = |V(G)|$ and $e(G):= |E(G)|$. 
For a positive integer $n$, we write $[n]:=\{1,\ldots,n\}$ and Call $G$ a \emph{graph on $[n]$} if $V=[n]$, or a \emph{graph on $n$ vertices} if $|G| = n$.
Throughout this paper we consider graphs up to isomorphism, in particular it applies to the definitions below of complete graph, star, path etc.
 The \emph{complete graph} $G = K_n$ is the graph on $[n]$ such that $E(K_n) = \{\{i,j\}\mid 1\leq i<j\leq n\}$. For $n\geq 2$, the \emph{star of order $n$} is the graph $St_n$ whose edge set is $E(St_n) = \{\{1,i\}\mid 2\leq i\leq n\}$. The star of order four is also called the \emph{claw}. For $n\geq 2$, the \emph{path} graph $P_n$ is the graph on $[n]$ with $E(P_n) = \{\{i,i+1\}\mid 1\leq i\leq n-1\}$. Clearly $St_n = P_n$ if and only if 
 $n\le 3$. Given graphs $G=(V,E)$ and $G'=(V',E')$, we say that $G'$ is a \emph{subgraph} of $G$, or $G$ \emph{contains} $G'$, if $V'\subseteq V$ and $E'\subseteq E$, in which case we write $G'\subseteq G$. A subgraph $G'\subseteq G$ is called an \emph{induced subgraph} of $G$ if $E(G') = E(K_{V(G')})\cap E(G)$, i.e. $G'$ has every edge in $G$ on the vertices of $G'$. For $n\geq 3$, the \emph{cycle} $G = C_n$ is the graph on $[n]$ with $E(G) = \{\{1,2\},\ldots,\{n-1,n\},\{1,n\}\}$. We say that a graph $G$ has an \emph{induced cycle} if $C_k$ is an induced subgraph of $G$ for some $k\geq 3$. A \emph{tree} is a connected graph with no induced cycles. A \emph{unicycle} graph is a connected graph $G$ with exactly one induced cycle. 
 A maximal connected subgraph of $G$ (in the sense of adding new edges and vertices) is called a (\emph{connected}) \emph{component} of $G$. $G$ is a \emph{forest} if every component of $G$ is a tree. Given a graph $G$, we say that a graph $G'$ is obtained from $G$ by \emph{adding a leaf to $G$} if $V(G') = V(G)\sqcup \{v\}$ and $E(G') = E(G)\cup \{\{a,v\}\}$ for some $a\in V(G)$. Given a graph $G$, we say that a graph $G'$ is obtained by \emph{adding a handle to $G$} if $V(G') = V(G)\sqcup \{v\}$ and there are vertices $a,b\in V(G)$ such that $E(G') = E(G)\sqcup \{\{a,v\},\{b,v\}\}$. We say that a graph $G'$ is obtained by \emph{adding a generalized handle to $G$} if for some $s\geq 2$ and $a,b\in V(G)$, we have $V(G') = V(G)\sqcup \{v_1,\ldots,v_s\}$ and $E(G') = E(G) \sqcup \{\{a,v_1\},\{v_1,v_2\},\ldots,\{v_{s-1},v_s\},\{v_s,b\}\}$.

The following small graphs will be of interest during various parts of this paper. 

\begin{figure}[htbp]
    \centering
    \caption{Some small graphs}\label{fig:small_graphs}
    \begin{subfigure}[b]{0.23\textwidth}
        \centering
        \begin{tikzpicture}[scale=0.8, every node/.style={circle, draw, minimum size=0.5cm, inner sep=1pt}]
\node(n1) at (140:2.2){1};
\node(n2) at (220:2.2){2};
\node(n3) at (0,0){3};
\node(n4) at (180:3.3){4};
\draw[black, thick](n1)--(n2);
\draw[black, thick](n1)--(n3);
\draw[black, thick](n2)--(n3);
\draw[black, thick](n1)--(n4);
\draw[black, thick](n2)--(n4);
        \end{tikzpicture}
        \caption*{A diamond}
        \label{fig:graph1}
    \end{subfigure}
    \hfill
    \begin{subfigure}[b]{0.23\textwidth}
        \centering
        \begin{tikzpicture}[scale=0.8, every node/.style={circle, draw, minimum size=0.5cm, inner sep=1pt}]
\node(n1) at (140:2){1};
\node(n2) at (220:2){2};
\node(n3) at (0,0){3};
\node(n4) at (40:2){4};
\node(n5) at (320:2){5};
\draw[black, thick](n1)--(n2);
\draw[black, thick](n1)--(n3);
\draw[black, thick](n2)--(n3);
\draw[black, thick](n3)--(n4);
\draw[black, thick](n3)--(n5);
        \end{tikzpicture}
        \caption*{A cricket}
        \label{fig:graph2}
    \end{subfigure}
    \hfill
    \begin{subfigure}[b]{0.23\textwidth}
        \centering
        \begin{tikzpicture}[scale=0.8, every node/.style={circle, draw, minimum size=0.5cm, inner sep=1pt}]
\node(n1) at (180:3){1};
\node(n2) at (220:2){2};
\node(n3) at (140:2){3};
\node(n4) at (0,1.3){4};
\node(n5) at (0,-1.3){5};
\draw[black, thick](n1)--(n2);
\draw[black, thick](n1)--(n3);
\draw[black, thick](n2)--(n3);
\draw[black, thick](n3)--(n4);
\draw[black, thick](n2)--(n5);
\draw[black, thick](n4)--(n5);
        \end{tikzpicture}
        \caption*{A house}
        \label{fig:graph3}
    \end{subfigure}
    \hfill
    \label{fig:four_graphs}
\end{figure}

\subsection{Binomial Edge Ideals}\label{sec-2.2}
Now we define the central objects to the study of this paper, the binomial edge ideal of a pair of graphs. 
An extensive treatment of some of their basic properties,  such as Cohen--Macaulayness and primary decomposition, can be found in the textbook \cite[Chapter 7]{HH-binom}.

Throughout the rest of this paper, we fix a field $F$. We assume when discussing binomial edge ideals $J_G$ that this field has characteristic zero or larger than $2^{|G|}$, and we assume when discussing binomial edge ideals $J_{G,H}$ of pairs of graphs $(G,H)$, that this field has characteristic zero or larger than $2^{|G||H|}$, with two exceptions: in the proof of Theorem \ref{thmA} (1) in Section \ref{sec5}, and in the proof of Theorem \ref{thmA} (2), we allow $F$ to be any field. We put these assumptions on the characteristic of $F$ in order to compute the analytic spread of ideals in a polynomial ring over $F$ using the theory of linear matroids, via~\cite[Theorem 8]{BMS}. We make use of this assumption at multiple critical points later in this section. Possibly these restrictions on the characteristic can be removed.

\begin{definition}[see \cite{Binom_gen}]\label{df1}
Let $m,n\geq 2$ be positive integers and consider the polynomial ring in $nm$ variables over $F$, 
    $S := F[x_{i,j}\mid 1\leq i\leq n,1\leq j\leq m]$.
We denote the quotient field of $S$ by $K = \qf(S) = F(x_{i,j}\mid 1\leq i\leq n,1\leq j\leq m)$.
    Given two graphs $G$ and $H$ on $m$ and $n$ vertices respectively, 
    the \emph{binomial edge ideal of the pair $(G,H)$} is the ideal
\begin{equation*}
    J_{G,H}:= (x_{i,k}x_{j,l}-x_{i,l}x_{j,k}\mid 1\leq i<j\leq n, 1\leq k<l\leq m, \{i,j\}\in E(G), \{k,l\}\in E(H))\subset S.
\end{equation*}

\end{definition}

\begin{remark}[{\cite[p.2]{Binom_gen}}]\label{rem:H=K_2}
    If $H = K_2$, and $G$ is a graph on $[n]$ then $J_{G,H}$ is the binomial edge ideal of $G$, which is defined as the ideal
    \begin{equation*}
        J_G = (x_iy_j-x_jy_i\mid \{i,j\}\in E(G))
        \subset F[x_1,\ldots,x_n,y_1,\ldots,y_n].
    \end{equation*}
    In particular, we have $\ell(J_{K_2,K_m}) = 2m-3$ for any integer $m\geq 3$ by \cite[Theorem 4.1]{LHVDRDOC}.
\end{remark}

\subsection{Matroids}\label{subsec:matroids}
Now we recall some elementary theory of matroids from \cite{Oxley} and \cite{kalai-symmetric_matroids} which will also be used in the proofs of Sections \ref{sec3}, and \ref{sec4}.

\begin{definition}[{c.f. \cite[p.7]{Oxley}}]
A \emph{matroid} is a pair $\mathcal{M} = (E, \mathcal{I})$ consisting of a finite set $E$, together with a family $\mathcal{I}$ of subsets of $E$ such that $\emptyset \in \mathcal{I}$, for every $I\in \mathcal{I}$ 
and $I'\subset I$, we have $I'\in \mathcal{I}$, and for $I_1,I_2\in \mathcal{I}$ with $|I_1|<|I_2|$, there exists an element $e\in I_2$ such that $I_1\cup \{e\}\in \mathcal{I}$. The elements of $\mathcal{I}$ are called \emph{independent sets}. A \emph{basis} is a maximal independent set under inclusion.  All bases of a matroid $\mathcal{M}$ have the same cardinality, and this cardinality is called the rank of $\mathcal{M}$, which is denoted by $\rk( \mathcal{M} )$.
Given two matroids 
$\mathcal{M}_1 = (E_1, \mathcal{I}_1)$ and $\mathcal{M}_2 = (E_2,\mathcal{I}_2$), a \emph{matroid isomorphism} $\Phi:\mathcal{M}_1\to \mathcal{M}_2$ is a bijection $\phi: E_1\to E_2$ such that 
$A\in \mathcal{I}_1$ if and only if $\phi(A)\in \mathcal{I}_2$.
\end{definition}

Whitney asked which matroids are isomorphic to \emph{linear} matroids in \cite[p.509]{Whitney-independence}, which we define now.

\begin{definition}
    A matroid $\mathcal{M} = ([n],\mathcal{I})$ is \emph{linear} (or \emph{representable}) if there exist
    a field $K$ and
    a matrix $A$ over $K$ with rows $A_1,\ldots,A_n$ such that $S\subset [n]$ is in $\mathcal{I}$ if and only if  $\{A_i\mid  i\in S\}$ is linearly independent (over $K$).
\end{definition}

\begin{definition}\label{df-algmat}
    Let $F\subset L$ be a finitely generated field extension and let $E\subset L$ be a finite set such that $L = F(E)$. Let $N = |E|$ and write $E = \{g_1,\ldots,g_N\}$. Let $\mathcal{I}$ be the set of subsets $S$ of $[N]$ such that $\{g_i\mid i\in S\}$ is algebraically independent over $F$. Then the pair $([N],\mathcal{I})$ forms a matroid (see for instance \cite[p.212, Theorem 6.7.1]{Oxley}), which we will denote by $\mathcal{M}_{E,F}$. We call $\mathcal{M}_{E,F}$ the \emph{algebraic matroid} associated to the field extension $F\subset F(E)$. A matroid $\mathcal{M}$ is called \emph{algebraic} if it is isomorphic to the matroid $\mathcal{M}_{E,F}$ of some finitely generated field extension $F\subset F(E)$.
\end{definition}

The following remark is a well known consequence of elementary field theory, showing that taking powers of generators in a field extension 
preserves the 
algebraic matroid.
\begin{remark}\label{rmk-powers_matroid}
    In the setting of Definition \ref{df-algmat}, take integers $i_1,\ldots,i_N\geq 1$ and denote $E' = \{g_1^{i_1},\ldots,g_N^{i_n}\}$. Then a set $\{g_j\mid j\in S\subset [N]\}\subset E$ is algebraically independent (over $F$) if and only if
    $\{g_j^{i_j}\mid j\in S\}\subset E'$ is algebraically independent. In particular the matroids $\mathcal{M}_{E,F}$ and $\mathcal{M}_{E',F}$ coincide.
\end{remark}

It is well-known that a linear matroid is algebraic, see for instance \cite[p.166, Proposition]{Ingleton}. The converse holds over characteristic zero, a result that goes back to Jacobi,  but is false in general, as shown by Ingleton~\cite[p.166, Example 15]{Ingleton} which exhibits a non representable algebraic matroid over $\mathbb{Z}/2\mathbb{Z}$.
The example of \cite{Ingleton} involves an algebraic matroid generated by homogeneous polynomials of different degrees but an equigenerated example follows immediately due to Remark \ref{rmk-powers_matroid}. 
We use a more recent result by Beecken, Mittmann and Saxena~\cite[Theorem 8]{BMS} which provides sufficient conditions on the field characteristic to imply that an algebraic matroid is also linear (over the same extension field). 
It will be used in proving Theorem \ref{thm-hyp} and Proposition \ref{propB}. We state some notation here to give a self contained description of their result which we will apply in Section \ref{sec3}.

First we fix some notation for the rest of this section. Let $m$ and $N$ be positive integers such that $m,N\geq 2$. Let $S = F[t_1,\ldots,t_m]$ be a polynomial ring, over a field $F$ in $m$ variables, and let $K$ be the quotient field of $S$, i.e. $K = F(t_1,\ldots,t_m)$. Let $E = \{g_1,\ldots,g_N\}$ be a finite subset of $K$. 
Following Definition \ref{df-algmat}, consider the algebraic matroid 
$\mathcal{M}_{E,F}$. Consider also the Jacobian matrix
\begin{equation*}
    \mathcal{J}: = (\partial g_i /\partial t_j)_{1\leq i\leq N\text{, }1\leq j\leq m}\in K^{N\times m}
\end{equation*}with rows $J_i := (\partial g_i /\partial t_1,\ldots, \partial g_i /\partial t_m)^t\in K^m$, $1\leq i\leq N$. Let $\mathcal{M}_{\mathcal{J}}$ be the linear matroid associated to the matrix $\mathcal{J}$. We will denote the derivative $\partial/\partial t_j$ by $\partial_j$ for $1\leq j\leq m$.

Recall that 
$\mathcal{M}_{E,F}$ is a matroid with ground set $[N]$ and a subset $X\subset [N]$ is independent if and only if the set $g_X:=\{g_i\mid i\in X\}$ is algebraically independent over $F$. Similarly, $\mathcal{M}_{\mathcal{J}}$ is a matroid with ground set $[N]$, where $X\subset [N]$ is independent if and only if $\mathcal{J}_X:= \{J_i\mid i\in X\}$ is linearly independent over $K$.

In \cite[Theorem 8]{BMS} they show that any algebraic matroid over a field of characteristic sufficiently large or zero is representable over the same field, which we paraphrase as follows.

\begin{lemma}[{c.f. \cite[Theorem 8]{BMS}, see also \cite[Proposition 2.4]{Ehrenborg-Rota}}]\label{lem2}
If $\text{char}(F)$ equals zero or is larger than $\max\{\deg(g_1),\ldots,\deg(g_N)\}^{\trdeg_F(F(g_1,\ldots,g_N))}$, then the matroids $\mathcal{M}_{E,F}$ and $\mathcal{M}_{\mathcal{J}}$ are isomorphic.
\end{lemma}

Here is 
another relationship between algebraic and linear matroids, also given in \cite{BMS}, which is valid over a field of any characteristic, and will be important for proving Theorem \ref{thmA}.

\begin{lemma}[{see \cite[Lemma 9]{BMS}}]\label{lem-ineq-any_characteristic}
    Let $F$ be any field, let $S = F[t_1,\ldots,t_m]$ be a polynomial ring over $F$, and let $g_1,...,g_N\in S$ be any polynomials. Then $\trdeg_F F(g_1,\ldots,g_N)$ is at least the rank of the Jacobian $\mathcal{J}=(\partial g_i/\partial t_j)_{1\leq i\leq N,1\leq j\leq m}$ of $g_1,\ldots,g_N$. 
    
    In particular, if a set of rows
    $\mathcal{J}_S$
    of 
    the Jacobian is linearly independent over $F(g_1,\ldots,g_N)$, then the corresponding set of polynomials $\{g_i\mid i\in S\}$ is algebraically independent over $F$.
\end{lemma}

\subsection{Hyperconnectivity} 
We review the hyperconnectivity matroid, which is a linear matroid that was introduced by Kalai in \cite{Kalai}. We first recall the definition of hyperconnectivity matrices.

\begin{definition}\label{def:hypconn}
Let $n$ and $d$ be positive integers, let $G$ be a graph on $[n]$, let $F$ be a field, and let $S = F[x_{1,1},\ldots,x_{d,n}]$ be a polynomial ring in $dn$ variables. For $1\leq i<j\leq n$ such that $\{i,j\}\in E(G)$, we define the vector
\begin{equation*}
    H_{i,j} = (0\ldots 0\text{ } x_{1,j}\text{ }x_{2,j}\ldots x_{d,j} \text{ }0\ldots 0\text{ } -x_{1,i}\text{ }-x_{2,i}\ldots -x_{d,i} \text{ }0\ldots 0)
    \in S^{nd}
\end{equation*}
where for $1\leq p\leq d$, $x_{p,j}$ is in the $(d(i-1)+p)$-th position, and $x_{p,i}$ is in the $(d(j-1)+p)$-th position. Consider the lexicographic ordering on the pairs $(i,j)\in \mathbb{Z}^2$, and let $H_1,\ldots,H_r$ be the vectors $H_{i,j}$ such that $\{i,j\}\in E(G)$ and $i<j$, arranged in this ordering. We consider the $d$-hyperconnectivity matrix of $G$, which is the matrix
\begin{equation*}
    \mathcal{H}_d(G):= \begin{pmatrix}
        H_1\\ H_2\\ \vdots\\ H_r
    \end{pmatrix}
    \in S^{e(G)\times d|G|}
\end{equation*}
A graph $G$ with $r$ edges is \emph{independent in the $d$-hyperconnectivity matroid} (or \emph{$d-$acyclic}) if the matrix $\mathcal{H}_d(G)$ has rank $r$, i.e. $\mathcal{H}_d(G)$ has independent rows. When $G=K_n$, we denote the resulting matroid by $\mathcal{H}^n_d$. In this paper we are especially interested in the matrix $\mathcal{H}(G) := \mathcal{H}_2(G)$, since $\mathcal{H}(G)$ is the Jacobian matrix of the binomial edge ideal $J_G$. 
\end{definition}

Kalai
introduced hyperconnectivity over a field of characteristic zero, however, this matroid was recently shown to be independent of the field characteristic in  \cite[Corollary 1.13]{tyomkyn-char_p_hyp}.

For $d=1$, a combinatorial characterization of  independence in $\mathcal{H}^n_d$ is well known:

\begin{remark}(c.f. \cite{Kalai})\label{rmk-1hc}
The $1$-hyperconnectivity matroid coincides with the graphic matroid. Thus
    a graph $G$ of order $n$ is independent in $\mathcal{H}^n_1$
    if and only if $G$ is a forest.
\end{remark}

The following theorem of Bernstein gives a combinatorial characterization of when a graph is independent in the $2$-hyperconnectivity matroid.

\begin{theorem}[{\cite[Theorems 4.2 and 4.4]{Bernstein}}]
Let $n\geq 2$ and let $G$ be a graph on $[n]$. $G$ is independent in the $2$-hyperconnectivity matroid if and only if there exists some acyclic orientation of G that has no alternating closed trail.
\end{theorem}
Here an alternating closed trail in a directed graph $D$ is a walk $w = (v_0,\ldots,v_k)$ such that each edge in $D$ appears at most once in $w$, $v_k = v_0$, and adjacent edges along $w$ have opposite orientations.


\subsection{Analytic Spread}

Now we recall some basic facts about analytic spread of ideals, which will be used throughout this paper. First we recall the definition of analytic spread. Recall that we have fixed a field $F$ in Section \ref{sec-2.2}. Further reading and unexplained definitions in commutative algebra can be found in \cite{atiyah-macdonald_com_alg_book, SH, Matsumura-com_ring_theory}.

\begin{definition}
    Let $S = F[x_1,\ldots,x_r]$ be a polynomial ring, and let $I\subset S$ be an ideal. Let $S[t]$ be a polynomial ring over $S$. The \emph{Rees algebra} of $I$ is the subring $R[It]:= \sum_{i\geq 0} I^it^i$ of $S[t]$. The quotient ring $\mathcal{F}(I):= R[It]/(x_1,\ldots,x_r)R[It]$ is called the \emph{special fiber ring} of $I$. The Krull dimension of $\mathcal{F}(I)$ is called the \emph{analytic spread} of $I$ and is denoted by $\ell(I)$. We shall emphasize the field $F$ we have fixed by saying that we are \emph{computing $\ell(I)$ over $F$}.
\end{definition}

The examples where analytic spread of $J_{G,H}$ is known are rather limited. In \cite{Binom_gen_paths} they compute the analytic spread of $J_{G,H}$ where $G$ is complete and $H$ is a path.

\begin{theorem}[{\cite[Corollary 3.6]{Binom_gen_paths}}]\label{spread_kn_path}
For $m,n\geq 2$, we have
    \begin{equation*}
    \ell(J_{K_m,P_n})=
    \begin{cases}
        n-1\textbf{ }& \text{ if }m=2\\
        mn-3\textbf{ }& \text{ if }m\geq 3.
    \end{cases}
\end{equation*}
\end{theorem}

Observe that the analytic spread of $J_{G,H}$ for a pair of graphs $(G,H)$, is unchanged upon reversing the order of the two graphs:

\begin{lemma}\label{lem-switch}
    For any graphs $G$ and $H$, the special fiber rings of $J_{G,H}$ and $J_{H,G}$ are isomorphic. In particular, $\ell(J_{G,H}) = \ell(J_{H,G})$.
\end{lemma}
\begin{proof}
Let $V(G)=[n]$ and $V(H)=[m]$, let $S = F[x_{i,j}\mid 1\leq i\leq n,1\leq j\leq m]$, 
$S' = F[x_{j,i}\mid 1\leq i\leq n,\leq j\leq m]$
and consider the $F$-algebras isomorphism $\sigma: S\to S'$ induced by $x_{i,j}\mapsto x_{j,i}$ for $1\leq i\leq n$ and $1\leq j\leq m$. 
Then $\sigma$ maps $J_{G,H}$ to $J_{H,G}$ and induces an isomorphism of their fiber rings.
\end{proof}

The following is shown during the proof of \cite[Proposition 4.8]{HeinzerKim} and gives the first step to express analytic spread as the rank of a linear matroid.

\begin{lemma}\label{lem0}
    Let $I\subset F[x_1,\ldots,x_t]$ an ideal which is generated by forms $E = \{g_1,\ldots g_N\}$ of a common degree. Then $\ell(I) = \trdeg_F F(g_1,\ldots,g_N)$, i.e. $\ell(I)$ is the rank of the algebraic matroid $\mathcal{M}_{E,F}$.
\end{lemma}

\begin{remark}\label{rmk-compute} 
    Suppose that $G$ and $H$ are graphs on $m$ and $n$ vertices respectively. Then the analytic spread $\ell(J_{G,H})$, computed over a field $F$ of characteristic zero or larger than $2^{|G||H|}$, only depends on the characteristic of $F$ (i.e. it is the same when computed over two fields of the same characteristic).
\end{remark}
\begin{proof}
Note that two fields have the same characteristic if and only if they contain isomorphic prime subfields. Let $F_0\subset F$ be a prime subfield. It is enough to show that $\ell(J_{G,H})$ computed over $F$ coincides with $\ell(J_{G,H})$ computed over $F_0$. Let $K_0 = F_0(x_{i,j}\mid 1\leq i\leq n,1\leq j\leq m)$ and recall that $K = F(x_{i,j}\mid 1\leq i\leq n,1\leq j\leq m)$. For clarity, we shall write $\ell_{F'}(J_{G,H})$ to emphasize we are computing the analytic spread over a particular field $F'$. Note that $\mathcal{H}(G,H)$ has entries in the subfield $K_0$ of $K$. Then by Proposition \ref{propB} and the fact that the rank of a matrix over a field $L$ is the same as its rank computed over any field extension of $L$, we have
    \begin{equation*}
        \ell_{F}(J_{G,H}) = \rk_{K}( \mathcal{H}(G,H)) = \rk_{K_0} (\mathcal{H}(G,H))  = \ell_{F_0}(J_{G,H}).
    \end{equation*}
\end{proof}

\section{Proof of Theorem \ref{thm-hyp} and Corollary \ref{cor-hyp}}\label{sec3}

In this section we give some applications 
of the representability results 
from \cite{BMS}.

\begin{proof}[proof of Theorem \ref{thm-hyp}]
Recall that $f_{i,j} = x_iy_j - x_jy_i$ for $1\leq i<j\leq n$. Let $N = \binom{n}{2}$, let $g_1 = f_{1,1},\ldots,f_{n-1,n}$ be the binomials $f_{i,j}$, $1\leq i<j\leq n$, arranged in lexicographical order, and take $E = \{g_1,\ldots,g_N\}$. 

Then the matrix $\mathcal{J}$ of Lemma \ref{lem2} is column equivalent to the matrix
$\mathcal{H}_2(G)$ of Definition~\ref{def:hypconn}. 
By Lemma \ref{lem0} and 
Lemma \ref{lem2}
\begin{equation*}
    \ell(J_G) = \rk (\mathcal{M}_{E,F}) = \rk (\mathcal{H}_2(G)).
\end{equation*}
\end{proof}

\begin{proof}[proof of Corollary \ref{cor-hyp}]
Let $G'$ be a subgraph of $G$ which is maximal having the property that the rows of $\mathcal{H}_2(G')$ are independent, and note how $\mathcal{H}_2(G')$ is a matrix consisting of a subset of the rows of $\mathcal{H}_2(G)$. Then by Theorem \ref{thm-hyp}
\begin{equation*}
    \ell(J_G) = \rk (\mathcal{H}_2(G) ) = \rk( \mathcal{H}_2(G')) = |E(G')|.
\end{equation*}
Now we are done, as \cite[Theorems 4.2 and 4.4]{Bernstein} together imply that for a subgraph $G'$ of $G$, the rows of $\mathcal{H}_2(G')$ are independent if and only if $G'$ has an acyclic orientation of its edges with no alternating closed trail.
\end{proof}

We conclude this section by stating some additional consequences of Theorem \ref{thm-hyp}.

Theorem \ref{thm-hyp} combined with \cite[Lemma 4.7]{Crespo-Santos} yields the exact value of the analytic spread of the complete bipartite graph $K_{n_1,n_2}$ on $n:=n_1+n_2$ vertices when computed over a field of characteristic zero or characteristic larger than $2^n$:
    \begin{equation*}
        \ell(J_{K_{n_1,n_2}}) = \begin{cases}
n_1n_2 \text{ }&\min\{n_1,n_2\}\leq 2\\
2n-4 &\min\{n_1,n_2\}\geq 2.
        \end{cases}
    \end{equation*}

Now we apply Theorem \ref{thm-hyp} and the gluing lemma \cite[Lemma 4.3]{Kalai} to prove that \cite[Question 5.5]{LHVDRDOC} has a positive answer when analytic spread is computed over a field of characteristic zero or sufficiently large.

\begin{remark}\label{prop5.1}
    Let $G$ be a graph and let $G'$ be a graph obtained by adding a handle to $G$. Then $\ell(J_{G'}) = \ell(J_G)+2$ when computed over a field of characteristic zero or characteristic larger than $2^{|G|}$.
\end{remark}
\begin{proof}
It is easy to see that $\rk(\mathcal{H}_2(G'))
    = \rk(\mathcal{H}_2(G))+2$, see e.g.~\cite[Lemma 4.3 (i)]{Kalai}, hence Theorem \ref{thm-hyp} yields
$\ell(J_{G'}) = \rk(\mathcal{H}_2(G'))
    = \rk(\mathcal{H}_2(G))+2
    = \ell(J_G)+2$.
\end{proof}

See Proposition \ref{prop-handle} for a generalization of Remark \ref{prop5.1} regarding pairs of graphs.

\section{Basic graph operations and \texorpdfstring{$\ell(J_{G,H})$}{l(J\_G,H)}}\label{sec4}

In this section we analyze the behavior of $\ell(J_{G,H})$ upon some basic operations on graphs, such as taking disjoint unions and augmentation of $G$ and $H$ by leaves and handles. 
The first remark worls over any field.

\begin{remark}\label{rmk-upper}
    For every field and every graphs $G$ and $H$, $\ell(J_{G,H})\leq |G||H|$.
\end{remark}

\begin{proof}
    We have that $J_{G,H}$ is an ideal in a polynomial ring $S$ in $|G||H|$ variables, so that $\ell(J_{G,H})\leq \dim(S) = |G||H|$.
\end{proof}

Recall that 
$\mathcal{H}(G,H)$ is the Jacobian of the binomial edge ideal $J_{G,H}$ where the rows and columns are given a fixed ordering. Now we observe a precise description of $\mathcal{H}(G,H)$, which will be useful in proving Theorem \ref{thmA} and Theorem \ref{thmA-2}.

\begin{notation}\label{constr-pair}
Let $n$ and $m$ be positive integers, let $G$ be a graph on $[n]$, and let $H$ be a graph on $[m]$. Let $S = F[x_{i,j}\mid 1\leq i\leq n, 1\leq j\leq m]$ be the polynomial ring in $nm$-many variables, and let $K = \qf(S) = F(x_{i,j}\mid 1\leq i\leq n, 1\leq j\leq m)$. For edges $e = \{i,j\}\in E(K_n)$ and $e' = \{k,l\}\in E(K_m)$ such that $i<j$ and $k<l$ define the binomial
\begin{equation*}
    f_{e,e'} = x_{i,k}x_{j,l}-x_{i,l}x_{j,k}
\end{equation*} so that $J_{G,H} = (f_{e,e'}\mid (e,e')\in E(G)\times E(H))$. For $e = \{i,j\}\in E(K_n)$ and $e' = \{k,l\}\in E(K_m)$ such that $i<j$ and $k<l$, define the row vector
\begin{equation*}
    \begin{split}
        H_{e,e'} &= \bigg(\frac{\partial}{\partial x_{1,1}}(f_{e,e'})\text{ }
    \frac{\partial}{\partial x_{1,2}}(f_{e,e'})\text{ }
    \cdots
    \frac{\partial}{\partial x_{1,m}}(f_{e,e'})\text{ }
    \frac{\partial}{\partial x_{2,1}}(f_{e,e'})\text{ }
    \cdots
    \frac{\partial}{\partial x_{n,m}}(f_{e,e'})\bigg)\\
    &=
    (0 \ldots 0\text{ } x_{j,l} \text{ } 0\ldots 0 \text{ } -x_{j,k} \text{ } 0 \ldots 0 \text{ }
    -x_{i,l} \text{ } 0 \ldots 0 \text{ }
    x_{i,k} \text{ }
    0 \ldots 0),
    \end{split}
\end{equation*}
and order these $H_{e,e'}$ lexicographically, prioritizing $e'$, namely:
\begin{equation*}
    \begin{split}
        H_{\{1,2\},\{1,2\}} &< H_{\{1,3\},\{1,2\}} <\ldots < H_{\{1,m\},\{1,2\}}<H_{\{2,3\},\{1,2\}}<\ldots <
    H_{\{n-1,n\},\{1,2\}}\\&< \ldots < H_{\{n-1,n\},\{m-1,m\}}.
    \end{split}
\end{equation*}
For graphs $G$ and $H$, let $H_1,\ldots,H_r$ be the rows $\{H_{e,e'}\mid e = \{i,j\}\in E(G), e' = \{k,l\}\in E(H), i<j, \text{ and } k<l\}$ ordered by $<$. Then define the following \emph{generalized hyperconnectivity matrix} of the pair $(G,H)$

\begin{equation}\label{hgh-df}
    \mathcal{H}(G,H) = \begin{pmatrix}
        H_1\\
        H_2\\
        \vdots \\
        H_r
    \end{pmatrix}\in K^{e(G)e(H)\times nm}.
\end{equation}
\end{notation}

Now we are ready to prove Proposition \ref{propB}, namely, that over suitable field characteristics
\begin{equation}\label{eq-pair}
   \ell(J_{G,H}) = \rk (\mathcal{H}(G,H)).
\end{equation}

\begin{proof}[proof of Proposition \ref{propB}]
Let $E=\{g_1,\ldots,g_r\}$ be the set of minimal binomial generators of $J_{G,H}$.
By Lemma \ref{lem0}, the rank of the algebraic matroid $\mathcal{M}_{E,F}$ of Lemma \ref{lem2} is exactly $\ell(J_{G,H})$.
By Remark~\ref{rmk-upper}, for characteristic zero or larger than $2^{|G||H|}$, we have by Lemma \ref{lem2}:
    \begin{equation*}
        \ell(J_{G,H}) = \rk(\mathcal{J}) = \rk (\mathcal{H}(G,H)).
    \end{equation*}
\end{proof}

Now we are ready to prove Corollary \ref{corr-mon}.
\begin{proof}[proof of Corollary \ref{corr-mon}]
    Using (\ref{hgh-df}), we see that $\mathcal{H}(G,H)$ is a submatrix of $\mathcal{H}(G',H')$.
    Then by Proposition \ref{propB},
    \begin{equation*}
        \ell(J_{G,H}) = \rk (\mathcal{H}(G,H))\leq \rk (\mathcal{H}(G',H')) = \ell(J_{G',H'}).
    \end{equation*}
\end{proof}

Next note that the analytic spread $\ell(J_{G,H})$ is additive along the components of $G$ and $H$:

\begin{remark}\label{rmk-comp}
    Let $G$ and $H$ be graphs and let $H_1,\ldots,H_r$ be the connected components of $H$. Then
    \begin{equation*}
        \ell(J_{G,H}) = \sum_{i=1}^r\ell(J_{G,H_i}).
    \end{equation*}
\end{remark}

\begin{proof}
   By (\ref{hgh-df}), we have the following diagonal block matrix expression for $\mathcal{H}(G,H)$ after possibly permuting the rows and columns of $\mathcal{H}(G,H)$:
    \begin{equation}\label{diag1} 
        \mathcal{H}(G,H) = \begin{pmatrix}
            \mathcal{H}(G,H_1) \text{ }& 0  \text{ }&\ldots \text{ }& 0 \text{ }& 0\\
            0 \text{ }& \mathcal{H}(G,H_2)  \text{ }&\ldots \text{ }& 0 \text{ }& 0\\
            \vdots \text{ }& \vdots  \text{ }& \text{ }& \vdots \text{ }& \vdots\\
            0 \text{ }& 0  \text{ }& \ldots \text{ }& \mathcal{H}(G,H_{r-1}) \text{ }& 0\\
            0 \text{ }& 0  \text{ }& \ldots \text{ }& 0 \text{ }& \mathcal{H}(G,H_r)
        \end{pmatrix}.
    \end{equation}
    Then Proposition \ref{propB} produces the following formula:
    \begin{equation*}
        \ell(J_{G,H}) = \rk (\mathcal{H}(G,H)) = \sum_{i=1}^r \rk (\mathcal{H}(G,H_i)) = \sum_{i=1}^r \ell(J_{G,H_i}).
    \end{equation*}
\end{proof}

Remark \ref{rmk-comp} and Lemma \ref{lem-switch} together 
imply
the following corollary, which 
reduces Problem~\ref{prob:char-l(J_{G,H})} to the case where both graphs $G$ and $H$ are connected.
\begin{corollary}\label{cor-comp2}
    Let $G$ and $H$ be graphs, let $H_1,\ldots,H_r$ be the connected components of $H$, and let $G_1,\ldots,G_s$ be the connected components of $G$. Then
    \begin{equation*}
        \ell(J_{G,H}) = \sum_{1\leq i\leq r\text{, }1\leq j\leq s}\ell(J_{G_j\text{, }H_i}).
    \end{equation*}
\end{corollary}

Our next application of Proposition \ref{propB} is given in the following Proposition, which demonstrates a sharp upper bound for the behavior of $\ell(J_{G,H})$ upon adding a leaf to $G$ or $H$.

\begin{proposition}\label{prop-leaf3}
    Let $G$ and $H$ be graphs and let $H'$ be a graph obtained by adding a leaf to $H$. Then
    \begin{equation*}
        \ell(J_{G,H'}) \leq  \ell(J_{G,H})+ \ell(J_G), 
    \end{equation*}
    with equality if $G$ is a tree.
\end{proposition}
\begin{proof}
Without loss of generality, we can assume $V(G) = [n]$, $V(H) = [m]$ and $H'$ is the graph obtained by adding the edge $e_0:=\{k,m+1\}$ to $H$.
By (\ref{hgh-df}), we have a decomposition
    \begin{equation}\label{eq-block}
        \mathcal{H}(G,H') = \begin{pmatrix}
            \mathcal{H}(G,H)\text{ }& 0\\
            A & M
        \end{pmatrix} \in K^{e(H')e(G)\times (nm+n)}
    \end{equation}
    where $(A \text{ }M)$ is the submatrix whose rows are the vectors $H_{e,e_0}$ such that $e\in E(G)$. 

Thus, $(A \text{ }M)$ is obtained from the $2$-hyperconnectivity matrix of $G$ by changing variables names and adding zero columns for the pairs of vertices not including $k$ nor $m+1$.

   Then Proposition \ref{propB} yields that
    \begin{equation*}
        \ell(J_{G,H'})  = \rk( \mathcal{H}(G,H'))\leq \rk (\mathcal{H}(G,H) )+ \rk ((A\text{ }M)) 
        =
        \ell(J_{G,H}) + \ell(J_G)
    \end{equation*}
    which completes the proof of the inequality.\\

Now suppose that $G$ is a tree. Note that $M$ is the 1-hyperconnectivity matrix of $G$. Being a tree, 
by Remark \ref{rmk-1hc},  
$\rk (M)=\rk (\mathcal{H}_1(G)) = e(G)=n-1$. Let $\phi: S\to S$ be the $F$-algebra map such that $\phi(x_{i,j}) = x_{i,j}$ for $j\leq m$ and equals zero otherwise and for any matrix $A$ over $K$, let $\phi(A)$ denote the matrix obtained by applying $\phi$ to each entry of $A$.
Then by Proposition \ref{propB}, we have,
    \begin{equation*}
        \ell(J_{G,H'}) \geq \rk(\phi(\mathcal{H}(G,H'))) = 
        \rk (\mathcal{H}(G,H)) + \rk (M) =  \ell(J_{G,H})+\ell(J_G)
    \end{equation*}
    which completes the proof.
    \end{proof}

The following corollary of Proposition \ref{prop-leaf3} will be useful in the proof of Theorem \ref{thmA} (2).

\begin{corollary}\label{cor-forest}
    Let $T$ and $T'$ be forests with $c_{T}$ and $c_{T'}$ many components resp. Then
    \begin{equation*}
        \ell(J_{T,T'}) = (|T|-c_T)(|T'|-c_{T'}).
    \end{equation*}
\end{corollary}
\begin{proof}
    By Corollary \ref{cor-comp2} we can assume that $T$ and $T'$ are trees, and now the result follows by induction from Proposition \ref{prop-leaf3}, Lemma \ref{lem-switch}, and the fact that $\ell(J_{K_2,K_2}) = \ell(J_{K_2}) = 1$.
\end{proof}

\begin{proposition}\label{prop-handle}
    Let $G$ and $H$ be graphs, and let $H'$ be a graph obtained by adding a handle to $H$. Then, $\ell(J_{G,H'}) \leq \ell(J_{G,H})+ 2\ell(J_G)$. Equality can fail even if $G$ and $H$ are both trees.
\end{proposition}

Before proving the inequality part we show an example that equality may fail.

\begin{example}
    Let $G$ be the star on five vertices, let $H$ be the path on three vertices, and let $H'$ be the triangle with a leaf. Note that $H'$ is obtained by adding a handle to $H$. 
    Calculations show that
    $\ell(J_{G,H}) = 8$, 
    $\ell(J_G) = 4$ 
    and $\ell(J_{G,H'})=14$.
    In particular, $\ell(J_{G,H'}) < 16 = \ell(J_{G,H}) + 2\ell(J_G)$.
\end{example}

\begin{proof}[proof of Proposition \ref{prop-handle}]
Write $V(G) = [n]$ and $V(H) = [m]$. Write $E(H') = E(H) \sqcup \{e,e'\}$. Then by (\ref{hgh-df}) there are matrices $H_1$ and $H_2$ which have isomorphic column spaces to $\mathcal{H}(G,e)$ and $\mathcal{H}(G,e')$, such that (after possibly permuting the columns of $\mathcal{H}(G,H)$) we have
    \begin{equation*}
        \mathcal{H}(G,H') = \begin{pmatrix}
            \mathcal{H}(G,H) \text{ } 0\ldots 0\\
            H_1\\
            H_2
        \end{pmatrix}
    \end{equation*}
    where $(0...0)\in K^{n}$.

    Then by subadditivity of rank and Proposition \ref{propB} we have,
    \begin{equation*}
        \begin{split}
            \ell(J_{G,H'}) &= \rk (\mathcal{H}(G,H'))
        \leq \rk (\mathcal{H}(G,H)) +\rk (\mathcal{H}(G,e)) + \rk (\mathcal{H}(G,e'))
        \\&= \ell(J_{G,H}) + \ell(J_{G,e}) + \ell(J_{G,e'})
        = \ell(J_{G,H})+ 2\ell(J_{G})
        \end{split}
    \end{equation*}
    since $J_{G_0,K_2} = J_{G_0}$ for any graph $G_0$.
\end{proof}

Combining Proposition \ref{prop-handle} with Proposition \ref{prop-leaf3} yields the following:

\begin{corollary}
    Let $G$ and $H$ be graphs, and let $H'$ be a graph obtained from $H$ by adding a generalized handle with $s\geq 2$ edges. Then, $\ell(J_{G,H'}) \leq \ell(J_{G,H})+ s\ell(J_G)$.
\end{corollary}

\section{Proof of Theorem \ref{thmA} (1)}\label{sec5}

In this section we prove Theorem \ref{thmA} (1). We begin by proving sharpness in its upper bound.
Recall that during this section we have fixed an arbitrary field $F$.

\begin{lemma}\label{lem-complete}
    For $m,n\geq 3$, $\ell(J_{K_n,K_m})=nm$.
\end{lemma}
\begin{proof}
    Let $J = J_{K_n,K_m}$. $J$ is the ideal of all 2 by 2 minors of a generic $n$ by $m$ matrix. So there is a parameterization of the affine variety $V$ whose coordinate ring $F[V]$ is the $F$-algebra generated by the minimal binomial generators of $J$, and said parameterization is given by the following surjective regular map of varieties
    \begin{equation*}
        K^{mn}\to V\subset K^{\binom{n}{2}\binom{m}{2}}
    \end{equation*}given by mapping an $n$ by $m$ matrix over $K$ to the vector of its 2 by 2 minors (in some fixed ordering). On the other hand, when restricted to the Zariski-dense open subset of matrices with rank at least three, the map $K^{mn}\to V$ has zero dimensional fibers by \cite[Theorem 1]{DOG}. Hence the image of this map has dimension $mn$. Then by Lemma \ref{lem0},
    \begin{equation*}
        mn = \dim V = \dim F[V] = \ell(J).
    \end{equation*}
\end{proof}

The proof of Lemma \ref{lem-complete} requires no assumption on the field $F$. It makes no reference to representability of the algebraic matroid of the field extension $F\subset F(J_{G,H})$.

Now we are ready to prove Theorem \ref{thmA} (1).

\begin{proof}[proof of Theorem \ref{thmA} (1)]
    The upper bound is given in Remark \ref{rmk-upper} and its tightness in Lemma \ref{lem-complete}. To see the lower bound, 

    take spanning forests $T\subset G$ and $T'\subset H$. 
    By Corollary \ref{cor-forest}, 
    $\ell(J_{T,T'})=(|T|-c_T)(|T'|-c_{T'})= \rk(\mathcal{H}(T,T'))$ (showing tightness, once the bound is proved). 
Being a submatrix, $\rk(\mathcal{H}(T,T'))\le \rk (\mathcal{H}(G,H))$. By Lemma~\ref{lem-ineq-any_characteristic}, 
$\rk (\mathcal{H}(G,H))\le \ell(J_{G,H})$, over any field, completing the proof of the lower bound.
\end{proof}

\section{Analytic Spread and Independence}\label{sec6}

In this section, we give a careful analysis of the pairs of graphs $(G,H)$ for which the analytic spread is as large as possible, i.e. $\ell(J_{G,H}) = e(G)e(H)$, which, by Proposition \ref{propB} (over suitable field characteristics), is equivalent to saying that the rows of the Jacobian matrix $\mathcal{H}(G,H)$ are linearly independent over the field $K = F(x_{i,j}\mid 1\leq i\leq |G|,1\leq j\leq |H|)$.\\

Corollary \ref{cor-forest} shows that 
independence holds for a pair of forests, 
and independence is characterized in 
the case that $H = K_2$ 
in Corollary \ref{cor-hyp}. 
Theorem \ref{thmA-2} shows that 
independence can fail 
even for pairs of graphs $(G,C_3)$ where $G$ contains a triangle with a leaf, while Theorem \ref{thmA} gives some classes of pairs of graphs $(G,H)$ where 
independence holds.

\subsection{General Criteria for 
independence}

We 
set some language for convenience:
    Let $G$ and $H$ be graphs on $[n]$ and $[m]$ respectively. Say that $(G,H)$ is  independent if $\ell(J_{G,H}) = e(G)e(H)$ and dependent if $\ell(J_{G,H}) < e(G)e(H)$. Say that $(G,H)$ is a basis if $(G,H)$ is independent and $nm = e(G)e(H)$. Say that $G$ is independent or dependent, or 
    special (respectively) if the rows of $\mathcal{H}_2(G)$ are linearly independent, dependent, or independent of cardinality $n$,
    respectively. By Theorem \ref{thm-hyp}, $G$ is 
    special
    if and only if $n = e(G) = \ell(J_G)$ if and only if $n = e(G)$ and there is some acyclic orientation of $G$ with no alternating closed trail.
Proposition \ref{propB} justifies this language:

\begin{remark}\label{rmk-depen}
    Let  $\mathcal{H}^{n,m}$ be the linear matroid associated to the matrix $\mathcal{H}(K_n,K_m)$. Then a pair of graphs $(G,H)$ on $n$ and $m$ vertices respectively, is a basis (when $m,n\ge 3$) if and only if $e(G)e(H) = nm$ and the rows of $\mathcal{H}(G,H)$ form a basis of the matroid $\mathcal{H}^{n,m}$, namely, 
    the rows of $\mathcal{H}(G,H)$ form a basis of $K^{mn}$. Similarly, $(G,H)$ is dependent (independent) if and only if the rows of $\mathcal{H}(G,H)$ form a dependent (independent) set in the matroid $\mathcal{H}^{n,m}$ (respectively).
\end{remark}
The following simple corollary will be useful and follows readily from Lemma \ref{lem2}.
\begin{lemma}\label{lem-subgr}
    Let $G\subset G'$ be graphs on $[n]$ and let $H\subset H'$ be graphs on $[m]$. Then the following statements hold, over characteristic zero or larger than $2^{nm}$.
    \begin{enumerate}
        \item[(1)] If $(G',H')$ is independent, then so is $(G,H)$.
        \item[(2)] If $(G,H)$ is dependent, then so is $(G',H')$. Furthermore,
        \begin{equation*}
            e(G)e(H) - \rk (\mathcal{H}(G,H)) \leq e(G')e(H') - \rk( \mathcal{H}(G',H')).
        \end{equation*}
        \item[(3)] $\rk (\mathcal{H}(G,H)) \le \rk( \mathcal{H}(G',H')).$
    \end{enumerate}
\end{lemma}
\begin{proof}
    By (\ref{hgh-df}), $\mathcal{H}(G',H')$ has the form
    \begin{equation*}
        \mathcal{H}(G',H')  = \begin{pmatrix}
            \mathcal{H}(G,H)\text{ }\text{ }0\\
            B
        \end{pmatrix}
    \end{equation*}
    for some matrix $B$. Thus if the rows of $\mathcal{H}(G',H')$ are linearly independent over $K$, then so are the rows of $\mathcal{H}(G,H)$, and (1) follows. (2) and (3) follow similarly from this decomposition.
\end{proof}
We now relate the dependence of a graph to the dependence of a pair.

\begin{lemma}
    Let $G$ and $H$ be graphs. If one of $G$ or $H$ is dependent, then so is $(G,H)$.
\end{lemma}
\begin{proof}

    Without loss of generality, $G$ is dependent, i.e. $\mathcal{H}_2(G)$ has dependent rows. Write $E(G) = \{e_1,\ldots,e_r\}$. 
    Using (\ref{hgh-df}) and Remark~\ref{rem:H=K_2} we see that for an edge $e\in E(H)$, 
    by renaming variables, 
    $\mathcal{H}_2(G)$ 
    is a 
submatrix $(H_{e_1,e} \ldots H_{e_r,e})^t$ of $\mathcal{H}(G,H)$ where the rest of the entries in each of its rows are zeros. Thus, these rows in $\mathcal{H}(G,H)$ are dependent, thus $(G,H)$ is dependent by Remark \ref{rmk-depen}.
\end{proof}

Next we exhibit a simple example where both $G$ and $H$ are independent, but $(G,H)$ is dependent.

\begin{example}\label{ex-diamond}

    Let $G = P_3$ be the path on three vertices and let $H$ be the diamond (see Figure \ref{fig:four_graphs}). Then we compute in Macaulay 2 (over $F = \mathbb{Q}$) that $\ell(J_{G,H}) = 9< 10 = e(G)e(H)$ so that $(G,H)$ is dependent. However, $G$ and $H$ are independent by Corollary~\ref{cor-hyp}.
\end{example}
This example shows the a diamond obstructs independence in general, as follows.
\begin{proposition}\label{prop-diamond}
    Let $n,m\ge 3$ and let $G$ and $H$ be graphs on $[n]$ and $[m]$ respectively such that $G$ is connected and some component of $H$ contains a diamond subgraph. Then $(G,H)$ is dependent.
\end{proposition}

\begin{proof}
    Since $G$ is connected and $n\geq 3$, there is a copy of $P_3$ as a subgraph $P\subset G$ of $G$. Let $D\subset H$ be a diamond subgraph of $H$. By Example \ref{ex-diamond}, $(P,D)$ is dependent. Then by Lemma \ref{lem-subgr}, $(G,H)$ is dependent.
\end{proof}

\subsection{Some questions about analytic spread}

In this subsection, we give two examples which evidence particular formulae and an inequality for the analytic spread of binomial edge ideals.

Now we give an example which shows that for graphs $G$ and two trees $H$ and $H'$ on the same number of vertices, the analytic spread of $J_{G,H}$ and $J_{G,H'}$ can differ, and in fact, 
in a predictable way, see Conjecture~\ref{conj:path_is_best_tree}.

\begin{example}\label{example_big}
We compute the analytic spread of $J_{G,H}$ for various trees $G$ and several choices of $H$. The way the analytic spread varies follows a uniform pattern as we will formalize in Conjecture \ref{conj:path_is_best_tree}. We define six trees on six vertices and compare their analytic spreads against various graphs $H$. Let $G_1 = P_6$, let $G_2$ be the tree with edge set $E(G_2) = \{\{1,2\},\{2,3\},\{3,4\},\{4,5\},\{2,6\}\}$, take $G_3$ to be the tree with edge set $E(G_3) = \{\{1,2\},\{2,3\},\{3,4\},\{4,5\},\{3,6\}\}$, let $G_4$ be the tree whose edge set is $E(G_4) = \{\{1,2\},\{2,3\},\{3,4\},\{2,5\},\{2,6\}\}$, take $G_5$ to be the tree with edge set\\ $E(G_5) = \{\{1,2\},\{2,3\},\{3,4\},\{2,5\},\{3,6\}\}$, and finally take $G_6 = St_6$ to be the star of order six.
Let $H_1$ be the graph obtained by gluing two disjoint four-cycles along an edge. Let $H_2$ be a $K_4$ with a handle.

We compute the following analytic spreads using random matrices in Macaulay2 to compute the rank of $\mathcal{H}(G,H)$ over characteristic zero.

    \begin{center}
\begin{tabular}{||c || c | c | c | c | c | c ||} 
 \hline
 the graph $H$ & $\ell(J_{G_1,H})$ & $\ell(J_{G_2,H})$ & $\ell(J_{G_3,H})$ & $\ell(J_{G_4,H})$ & $\ell(J_{G_5,H})$ & $\ell(J_{G_6,H})$ \\ [0.5ex] 
 \hline\hline
 $C_3$              &15 &14 &14 &13 &13 &12 \\ 
 \hline
 $C_3$ and a leaf   &20 &19 &19 &18 &18 &17 \\
 \hline
  $C_4$              &20 &19 &19 &18 &18 &17 \\
 \hline
 $C_5$              &25 &25 &25 &25 &25 &24 \\   
 \hline
 $C_6$              &30 &30 &30 &30 &30 &29 \\   
 \hline
 a cricket          &25 &24 &24 &23 &23 &22 \\
 \hline
 a house            &27 &26 &26 &25 &25 &24 \\
 \hline
 a diamond          &21 &20 &20 &19 &19 &18 \\
 \hline
 $H_1$              &32 &31 &31 &30 &30 &29 \\
 \hline
 $H_2$              &27 &26 &26 &25 &25 &24 \\
 \hline
 $K_6$              &33 &32 &32 &31 &31 &30 \\
 \hline
\end{tabular}
\end{center}
\end{example}

Example \ref{example_big} leads to the following 
conjecture.
\begin{conjecture}\label{conj:path_is_best_tree}
    Let $n,m\geq 2$ be positive integers. Let $T_1$ and $T_2$ be trees on $n$ and $m$ vertices respectively. Fix $i\in V(T_1)$ and $j\in V(T_2)$. Let $G_1$ be the tree on $n+m+l$ vertices whose edge set is $E(T_1)\sqcup E(T_2)\sqcup \{\{i,m+n+1\},\{m+n+1,m+n+2\},\ldots,\{m+n+l,j\}\}$ and let $G_2$ be the tree on $n+m+l$ vertices 
    obtained from the disjoint union of $T_1$ and $T_2$  by identifying $i$ with $j$ and then add the "tail" of edges 
    $\{i,n+m+1\}, \{n+m+1,n+m+2\} ,\ldots,\{m+n+l,m+n+l+1\}$. Then for any 
    graph $H$,
    $$\ell(J_{G_1,H}) \ge \ell(J_{G_2,H}).$$
\end{conjecture}
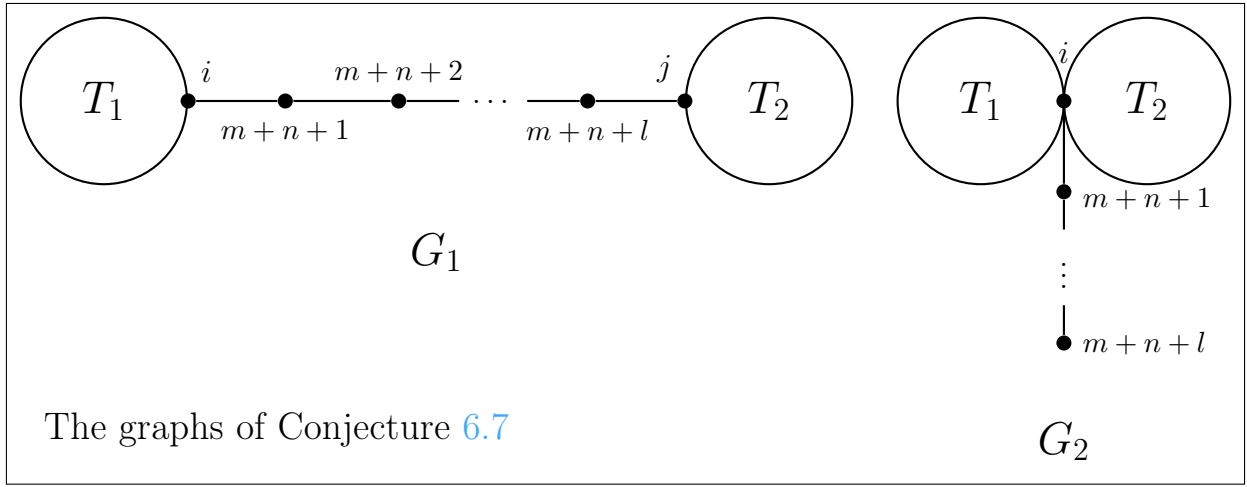
\label{question1}

\begin{tikzpicture}[
    blob/.style={draw, circle, minimum size=2.2cm, thick},
    dot/.style={fill, circle, inner sep=2pt}
]


\node[blob] (T1) at (0,0) {\Large $T_1$};
\node[blob] (T2) at (8.8,0) {\Large $T_2$}; 

\node[dot, label={75:$i$}] (i) at (T1.east) {};
\node[dot, label={below:{\small $m+n+1$}}] (n1) at (2.4,0) {};
\node[dot, label={above:{\small $m+n+2$}}] (n2) at (3.9,0) {};

\node[dot, label={below:{\small $m+n+l$}}] (nl) at (6.4,0) {};
\node[dot, label={105:$j$}] (j) at (T2.west) {};

\draw[thick] (i) -- (n1) -- (n2);
\draw[thick] (n2) -- (4.7,0);
\node at (5.15,0) {$\dots$};
\draw[thick] (5.6,0) -- (nl);
\draw[thick] (nl) -- (j);

\node (G1) at (4.4,-2) {\Large $G_1$};


\begin{scope}[shift={(12.7,0)}] 
    \node[draw, circle, minimum size=2.2cm, thick] (T1_G2) at (-1.1,0) {\Large $T_1$};
    \node[draw, circle, minimum size=2.2cm, thick] (T2_G2) at (1.1,0) {\Large $T_2$};

    \node[dot, label={[label distance=7pt]above:$i$}] (i_G2) at (0,0) {};

    \node[dot, label={[yshift=-2.5pt]right:{\small $m+n+1$}}] (n1_G2) at (0,-1.2) {};
    \node (dots_G2) at (0,-2.2) {$\vdots$};
    \node[dot, label=right:{\small $m+n+l$}] (nl_G2) at (0,-3.2) {};

    \draw[thick] (i_G2) -- (n1_G2);
    \draw[thick] (n1_G2) -- (0,-1.7);
    \draw[thick] (0,-2.7) -- (nl_G2);

    \node (G2) at (0,-4.5) {\Large $G_2$};
\end{scope}


\coordinate (BoxSW) at ($(current bounding box.south west)+(-5pt,-5pt)$);
\coordinate (BoxNE) at ($(current bounding box.north east)+(5pt,5pt)$);

\node[anchor=south west, xshift=10pt, yshift=10pt] at (BoxSW) {\large{The graphs of Conjecture \ref{conj:path_is_best_tree}}};

\draw[thin] (BoxSW) rectangle (BoxNE);

\end{tikzpicture}

This conjecture implies that $\ell(J_{P_n,H}) \ge \ell(J_{T,H})$ for every $n$-vertex tree $T$ and every graph $H$ (recall that $P_n$ is the path on $n$ vertices).

Computations of the    
    analytic spreads for pairs of cycles, using random matrices in Macaulay2 to compute the rank of $\mathcal{H}(G,H)$ suggest the following.

\begin{conjecture}\label{question-cycle}
    Let $n,m\geq 3$ be positive integers. 
    Then if $F$ has characteristic zero
    \begin{equation*}
        \ell(J_{C_n,C_m}) =
        \begin{cases}
            nm \text{ }& \text{ $n$ and $m$ are odd}\\
            nm - 1 & \text{ exactly one of $n$ or $m$ is even}\\
            nm-2 & \text{ $n$ and $m$ are even}.
        \end{cases}
    \end{equation*}
    This identity implies that $(C_n,C_m)$ is a basis if and only if $n$ and $m$ are odd.
\end{conjecture}

Note that Conjectures~\ref{conj:path_is_best_tree} and  \ref{question-cycle} together with  Theorem~\ref{thmA-2}(4) 
and the fact that $\ell(J_{P_n,K_m})=nm-3$ for $n,m\ge 3$ by  \cite[Corollary 3.6 (d)]{Binom_gen_paths} 
would settle Conjecture~\ref{conj:odd_cycles_basis} on a characterization of basis pairs of graphs.

\section{Proof of Theorem \ref{thmA} (2)}\label{sec-7:prfA_pt1}

In this section we complete the proof of Theorem \ref{thmA} by proving part (2), which states that for a unicycle graph $H$, $J_{P_n,H}$ has the maximum possible analytic spread of $(n-1)|H|$. We begin by first proving a sequence of two technical lemmas, accompanied by some remarks. Recall that we have fixed an arbitrary field $F$. Throughout this section, also fix positive integers $n\geq 4$ and $m\geq 3$. Recall that we denote $S = F[x_{i,j}\mid 1\leq i\leq n, 1\leq j\leq m]$ and we take $K$ to be the quotient field of $S$. Let $N = (n-1)m$. We begin with a series of lemmas which analyze the matrix $\mathcal{H}(P_n,C_m)$. The following lemma is clear by the permutation expansion of the determinant and will be indispensable to our proof.

\begin{lemma}\label{lem-det}
    Let $M\in S^{N\times N}$ be a matrix and write $M = (m_{i,j})_{1\leq i\leq N,1\leq j\leq N}$. Suppose that
    \begin{enumerate}
        \item[(1)] every entry of $M$ lies in $\{0\}\cup \{\pm x_{i,j}\mid 1\leq i\leq n, 1\leq j\leq m\}$, and 
        \item[(2)] there exists a monomial $P\in S$ for which there is a unique permutation $\sigma$ of $[N]$ having the property that
        \begin{equation*}
            \prod_{i=1}^N m_{i,\sigma(i)} \in \{P,-P\}.
        \end{equation*}
    \end{enumerate}
    Then $M$ is invertible.
\end{lemma}

Now we introduce some additional notation for the rest of this section. Let $\mathcal{H} = \mathcal{H}(P_n,C_m)$ and for $1\leq i\leq nm$ let $\mathcal{H}^i$ be the $i$-th column of $\mathcal{H}$. Let $M = (\mathcal{H}^2 \text{ }\mathcal{H}^3\ldots \mathcal{H}^{(n-1)m+1})$. That is, $M$ is the submatrix of $\mathcal{H}$ obtained by removing its first column and last $m-1$ columns.

The entries of $M$ are then given by the convention (\ref{hgh-df}) for $\mathcal{H}(P_n,C_m)$. For the square matrix $M\in S^{N\times N}$ write $M = (m_{i,j})_{1\leq i\leq N,1\leq j\leq N}$. Consider the monomial of degree $N$
\begin{equation*}
    P = \bigg(\prod_{2\leq i\leq n, 1\leq j\leq m-1}(-x_{i,j})\bigg)\bigg(\prod_{1\leq i\leq n-1}-x_{i,m}\bigg).
\end{equation*}
and the permutation $\sigma$ of $[N]$ defined by
\begin{equation*}
    \sigma(i) = \begin{cases}
        i-1\text{ }& i-1\neq 0\text{ modulo $m$}\\
        i+m-1 \mod N &\text{otherwise}.
    \end{cases}
\end{equation*}


The next remark describes a convenient block matrix expression for $\mathcal{H}$ by size $m$ square blocks, which follows from (\ref{hgh-df}) and will be useful in our analysis on permutations of the columns of $M$, since $M = (\mathcal{H}^2 \text{ }\mathcal{H}^3\ldots \mathcal{H}^{(n-1)m+1})$. Therein, we shall apply a specific column permutation $\lambda$ to $\mathcal{H}$ to obtain a convenient expression for $\mathcal{H}$.

\begin{remark}\label{rmk-M}We have the following statements regarding the matrix $\mathcal{H}$.
    \begin{enumerate}
        \item[(i)] For $1\leq i\leq n$ we define the following
size $m$ square matrices
    \begin{equation*}
        Q_i := Q_{m,i} = \begin{pmatrix}
             x_{i,m}\text{ }&\text{ }&\text{ }&\text{ }&\text{ }& -x_{i,1} \\
            x_{i,2} &-x_{i,1} \text{ }&\text{ }&\text{ }&\\
            & x_{i,3} &-x_{i,2}&&\\
            &&\ddots\text{ }&\text{ }&\\
            &&&x_{i,m-1}& -x_{i,m-2}\\
            &&&&x_{i,m}&-x_{i,m-1}
        \end{pmatrix}\in K^{m\times m}.
    \end{equation*}
    Consider the permutation $\lambda\in S_{mn}$ given by $\lambda(j) = j-1$ if $m\nmid j-1$ and $\lambda(j) = j+m-1$ otherwise. $\lambda$ is simply the switching of the first and last columns of each $mn\times m$ block of $\mathcal{H}$, or equivalently switching the first and last column of each matrix $Q_i$. 
    $\mathcal{H}$ admits the following expression due to (\ref{hgh-df}), after applying the permutation $\lambda$ to the columns of $\mathcal{H}$
    \begin{equation*}
        \mathcal{H}= \mathcal{H}(P_n,C_m)
        =
        \begin{pmatrix}
            Q_2 \text{ }& -Q_1 \text{ }&  \text{ }&  \text{ }&  \text{ }&  \text{ }& \\
              &Q_3&  -Q_2\\
              &&Q_4& -Q_3\\
             &&&Q_5& -Q_4\\
             &&&&&\ddots\\
             &&&&&Q_n& -Q_{n-1}
        \end{pmatrix}
        \in K^{(n-1)m\times nm}.
    \end{equation*}
    \item[(ii)] In particular, the monic parts of the monomials in each row of $M$ are distinct. Thus if $\tau$ is a permutation of $[N]$ such that $\prod_{i=1}^Nm_{i,\tau(i)}$ is nonzero (or equivalently, $m_{i,\tau(i)}\neq 0$ for all $i$), then if for some $i$, $m_{i,\tau(i)} = m_{i,\sigma(i)}$, we must have $\tau(i) = \sigma(i)$.
    \end{enumerate}
\end{remark}

We obtain the following statement using Remark \ref{rmk-M}.

\begin{remark}\label{rmk-sigma}
By construction of $\sigma$, $M$, and $\mathcal{H}(P_n,C_m)$, the monomial $P$ above equals 
    \begin{equation*}
        P = \prod_{i=1}^N m_{i,\sigma(i)}.
    \end{equation*}
\end{remark}

Remark \ref{rmk-sigma} in part states that there is a permutation $\tau$ for which $\prod_{i=1}^N m_{i,\tau(i)}\in\{P,-P\}$, namely the permutation $\sigma$. We shall prove that $\sigma$ is the only permutation for which this happens.


\begin{lemma}\label{lem-sigma}
    $\sigma$ is the unique permutation $\tau$ of $[N]$ for which $\prod_{i=1}^N m_{i,\tau(i)} \in \{P,-P\}$.
\end{lemma}

Before proving Lemma \ref{lem-sigma}, we prove one additional remark, which states a series of restrictions on permutations $\tau$ satisfying the condition of Lemma \ref{lem-sigma}.

\begin{remark}\label{rmk-perm_facts}
   Let $\tau$ be a permutation of $[N]$ for which $\prod_{i=1}^N m_{i,\tau(i)} \in \{P,-P\}$. Then the following statements hold.
    \begin{enumerate}
        \item[(I)] $m_{i,\tau(i)}\neq \pm m_{j,\tau(j)}$ whenever $i\neq j$.
        \item[(II)] $m_{i,\tau(i)}\notin \{\pm x_{1,j}\mid 1\leq j\leq m-1\}$ for $1\leq i\leq N$.
        \item[(III)] If $1\leq j\leq N$ and for some $i\in\{1,\ldots,N\}$ we have $m_{i,\tau(i)} = m_{i,j}$, then for $i'>i$, we have $m_{i',\tau(i')}\neq m_{i',j}$. That is, the set $\{m_{i,\tau(i)}\mid 1\leq i\leq N\}$ only intersects the entries of a given column of $M$ once.
        \item[(IV)] $m_{i,\tau(i)}\neq 0$ for $1\leq i\leq N$. 
    \end{enumerate}
\end{remark}

\begin{proof}
    (I) follows from the assumption that $\prod_{i=1}^N m_{i,\tau(i)} \in \{P,-P\}$, and the fact that $P$ is square-free. (II) holds since no variable of the form $x_{1,j}$, $1\leq j\leq m-1$ divides $P$. (III) follows from the fact that $\tau$ is a permutation, and is consequently injective. (IV) follows from 
    the fact that $P,-P\neq 0$.
\end{proof}

Now we are ready to prove Lemma \ref{lem-sigma}.

\begin{proof}[proof of Lemma \ref{lem-sigma}] Note that, since $(m_{i,j})_{1\leq i,j\leq N} = M = (\mathcal{H}^2 \mathcal{H}^3\ldots \mathcal{H}^{(n-1)m+1})$, $m_{i,j}$ lives in the $i$-th row and $(j+1)$-st column of $\mathcal{H}$. This fact, in combination with Remark \ref{rmk-M} (i) will be important throughout the proof to keep track of the possible values of $m_{i,\tau(i)}$ for each $i$, for a permutation $\tau$. Suppose that $\tau$ is a permutation of $[N]$ such that $\prod_{i=1}^N m_{i,\tau(i)} \in \{P,-P\}$. We will show that $\tau = \sigma$ by induction on positive integers $v$ for which $\tau(i) = \sigma(i)$ for $i\leq vm$. Note that by the Remark \ref{rmk-M} (ii) it is enough to show that $m_{i,\tau(i)} = m_{i,\sigma(i)}$ for all $i$. The $v=1$ case has a slightly different proof to the others, so we first prove the $v=1$ and $2$ cases.

Now we proceed by analyzing the possible values of $m_{i,\tau(i)}$ for each $i$. By Remark \ref{rmk-M} (i) and the construction of $M$, we have $m_{1,\tau(1)}\in \{x_{2,1},-x_{1,m},-x_{1,1}\}$, while by Remark \ref{rmk-perm_facts} (II), $m_{1,\tau(1)}\neq -x_{1,1}$.  If $m_{1,\tau(1)} = x_{2,1}$, then by Remark \ref{rmk-M} (i), $m_{2,\tau(2)}\in \{-x_{2,1},-x_{1,2},x_{1,1}\}$, contradicting Remark \ref{rmk-perm_facts} (I) and (II). Thus, $m_{1,\tau(1)} = -x_{1,m} = m_{1,m} = m_{1,\sigma(1)}$.

By Remark \ref{rmk-M} (i), $m_{2,\tau(2)}\in \{-x_{2,1},-x_{1,2},x_{1,1}\}$.  Then by Remark \ref{rmk-perm_facts} (II), $m_{2,\tau(2)}$ is forced to equal $-x_{2,1} = m_{2,1} = m_{2,\sigma(2)}$. Now we will show that $m_{i,\tau(i)} =  m_{i,\sigma(i)}$ for $i = 3,\ldots,m$ by induction. If $3\leq k\leq m$ and $m_{i,\tau(i)} = m_{i,\sigma(i)}$ for $i< k$, then by Remark \ref{rmk-M} (i) and Remark \ref{rmk-perm_facts} (II), and (III), we must have that $m_{k,\tau(k)} = m_{k,\sigma(k)}$ (the other options are either in the same row as $m_{i,\tau(i)}$ for some $i<k$ or are nondivisors of $P$). Now we have shown that, $m_{i,\tau(i)} = m_{i,\sigma(i)}$ for $1\leq i\leq m$.

So far we have proven the $v=1$ base case of our induction. Next we prove the $v = 2$ case. Firstly, by Remark \ref{rmk-M} (i), we have that $m_{m+1,\tau(m+1)} \in \{x_{3,m},x_{3,1},-x_{2,m},-x_{2,1}\}$. By Remark \ref{rmk-perm_facts} (I) and the fact that $m_{2,\tau(2)} = -x_{2,1}$, we cannot have $m_{m+1,\tau(m+1)} = -x_{2,1}$. Also $m_{m+1,\tau(m+1)}\neq x_{3,1}$, otherwise, the only possible values of $m_{3,\tau(3)}$ would contradict Remark \ref{rmk-perm_facts} (I). In addition, Remark \ref{rmk-perm_facts} (III) and $m_{1,\tau(1)} = m_{1,m}$ implies that $m_{m+1,\tau(m+1)}\neq x_{3,m}$. Hence, $m_{m+1,\tau(m+1)} = -x_{2,m} = m_{m+1,2m} = m_{m+1,\sigma(m+1)}$.

Next we assess the value of $m_{m+2,\tau(m+2)}$. Note that by Remark \ref{rmk-perm_facts} (I) implies that $m_{i,\tau(i)}\notin\{\pm x_{2,i}\mid 1\leq i\leq m\}$ for $i>m+1$. Thus the only possible values of $m_{m+2,\tau(m+2)}$ are $x_{3,2}$ and $-x_{3,1}$. On the other hand, $m_{m+2,\tau(m+2)}\neq m_{m+2,m} =x_{3,2}$ by Remark \ref{rmk-perm_facts} (III) and the fact that $m_{1,\tau(1)} = m_{1,m}$. Thus $m_{m+2,\tau(m+2)} = -x_{3,1} = m_{m+2,m+2} = m_{m+2,\sigma(m+2)}$. Now by induction on $i$ and the fact that $m_{i,\tau(i)}\notin\{\pm x_{2,i}\mid 1\leq i\leq m\}$ for $i>m+1$, Remark \ref{rmk-perm_facts} forces that $m_{i,\tau(i)} = m_{i,\sigma(i)}$ for $m+3\leq i\leq 2m$. This completes the proof of the $v=2$ case.

Assume that $v\geq 2$ and $m_{i,\tau(i)} = m_{i,\sigma(i)}$ for $1\leq i\leq vm$. Now we shall prove that $m_{i,\tau(i)} = m_{i,\sigma(i)}$ for $vm+1\leq i\leq (v+1)m$. We are assuming in part that
\begin{equation*}
    m_{(v-1)m+1,\tau((v-1)m+1)} = m_{(v-1)m+1,\sigma((v-1)m+1)} = m_{(v-1)m+1,vm} = -x_{v,m},
\end{equation*}
while $m_{(v-1)m+1,vm} = -x_{v,m}$ is in the same column of $M$ as the $x_{v+2,m}$ in the $(vm+1)$-th row of $M$. Then Remark \ref{rmk-perm_facts} (III) implies that $m_{vm+1,\tau(vm+1)}\neq x_{v+2,m}$. On the other hand, our induction hypothesis also implies that
\begin{equation*}
    m_{(v-1)m+2,\tau((v-1)m+2)} = m_{(v-1)m+2,(v-1)m+2} = -x_{v+1,1}
\end{equation*}
so that $m_{vm+1,\tau(vm+1)}\neq x_{v+1,1}$ (by Remark \ref{rmk-perm_facts} (I)). Now using Remark \ref{rmk-M} (i) we are only left with the possibility
\begin{equation*}
    m_{vm+1,\tau(vm+1)}\in \{x_{v+2,1},-x_{v+1,m}\}.
\end{equation*}
If $m_{vm+1,\tau(vm+1)} = x_{v+2,1}$, then (using Remark \ref{rmk-perm_facts} (III) again)\\ $m_{vm+2,\tau(vm+2)}\in \{-x_{v+2,1},-x_{v+1,2},x_{v+1,1}\}$, contradicting Remark \ref{rmk-perm_facts} (I), since our induction hypothesis implies that there are $i<vm+2$ such that $m_{i,\tau(i)}$ take values $\pm x_{v+1,j}$ for $j=1,2$. Thus $m_{vm+1,\tau(vm+1)} = -x_{v+1,m}$.

Note that our induction hypothesis implies that for $1\leq j\leq m-1$ there is a $k$ such that $m_{k,\tau(k)}\in \{\pm x_{v+1,j}\}$ (namely $k := (v-1)m+j+1$). Now by Remark \ref{rmk-perm_facts} (I) and induction on $j\in \{2,\ldots, m\}$, we have $m_{vm+j,\tau(vm+j)} = x_{v+2,j} = m_{vm+j,\sigma(vm+j)}$ for $j = 2,\ldots,m$. So we have proven that $m_{i,\tau(i)} = m_{i,\sigma(i)}$ for $vm+1\leq i\leq (v+1)m$. This completes our induction on $v$. Thus, $\tau(i) = \sigma(i)$ for $1\leq i\leq (n-1)m = N$. As in, $\tau = \sigma$,  completing the proof.
\end{proof}

\begin{proof}[proof of Theorem \ref{thmA} (2)]

    By Lemma \ref{lem-det} and Lemma \ref{lem-sigma}, the matrix $M$ is invertible. Thus, $\mathcal{H} = \mathcal{H}(P_n,C_m)$ contains an invertible submatrix of size $(n-1)m$ and so has rank at least $(n-1)m$. On the other hand, $\mathcal{H}(P_n,C_m)$ is an $(n-1)m$ by $nm$ matrix and so its rank does not exceed $(n-1)m$. Hence, $\rk (\mathcal{H}(P_n,C_m))=(n-1)m$. 
Then by Lemma~\ref{lem-ineq-any_characteristic}, 
$\ell(J_{P_n,C_m}) \ge \rk (\mathcal{H}(P_n,C_m)) = (n-1)m = |E(P_n)||E(C_m)|$. However, clearly the transcendence degree of $F\subset F(f_{e,e'}\mid e\in E(P_n), e'\in C_m)$ does not exceed the number of generators hence also $\ell(J_{P_n,C_m}) \le |E(P_n)||E(C_m)|$, completing the proof.
\end{proof}

\section{Proof of Theorem \ref{thmA-2} (1)}\label{sec-8:prfA_part2}

The goal of this section is to prove Theorem \ref{thmA-2} (1), which states that for a pair of connected graphs $(G,H)$ on at least three vertices such that one graph contains a star of larger order than some cycle in the other graph, we have $\ell(J_{G,H})<|E(G)||E(H)|$.

Using 
Corollary \ref{corr-mon} and Lemma \ref{lem-switch} 
we reduce the problem to the case when $G$ is an $n$-star and $H$ is an $(n-1)$-cycle, which we prove in Lemma \ref{lem-star}. 
To prove this lemma, we describe an explicit large enough subspace of the right kernel of $\mathcal{H}(G,H)$, which then implies that its left kernel is nonzero, namely that $(G,H)$ is dependent.

\begin{lemma}\label{lem-star}
    Let $n\geq 4$ be a positive integer. Let $G = St_n$ be the star graph of order $n$ and let $H = C_{n-1}$ be the cycle of order $n-1$. Then the following holds:
    \begin{enumerate}
        \item[(a)]  $\rk (\mathcal{H}(G,H)) \leq e(G)e(H) - 1$, and 
        \item[(b)] if $n$ is odd, then $\rk (\mathcal{H}(G,H)) \leq e(G)e(H) - 2$.
    \end{enumerate}
\end{lemma}

\begin{proof}
    For $1\leq i\leq n$ and $3\leq r\leq n-1$, consider the matrix
    \begin{equation*}
        B_{r,i} = \begin{pmatrix}
            x_{i,2} \text{ }& -x_{i,1}\text{ }&\text{ }&\text{ }&\text{ }& \\
            & x_{i,3} &-x_{i,2} \text{ }&\text{ }&\text{ }&\\
            &&&\ddots\text{ }&\text{ }&\\
            &&&&x_{i,r}& -x_{i,r-1}\\
            x_{i,r}&&&&&-x_{i,1}
        \end{pmatrix}\in K^{r\times r}.
    \end{equation*}
    And take $B_i = B_{n-1,i}$ for $1\leq i\leq n$. Using (\ref{hgh-df}) we compute that
    \begin{equation*}
        \mathcal{H}:= \mathcal{H}(G,H)
        =
        \begin{pmatrix}
            B_2 \text{ }& -B_1 \text{ }&  \text{ }&  \text{ }&  \text{ }&  \text{ }& \\
            B_3  &&  -B_1\\
            B_4  &&& -B_1\\
            B_5 &&&& -B_1\\
            \vdots &&&&&\ddots\\
            B_n &&&&&& -B_1
        \end{pmatrix}
        \in K^{(n-1)^2\times n(n-1)}.
    \end{equation*}
Now we shall construct vectors which we will show are in the right kernel of $\mathcal{H}$. Let
\begin{equation}
    \alpha = \begin{pmatrix}
        x_{1,1}\\ x_{1,2}\\ \vdots \\ x_{1,n-1}
    \end{pmatrix}
\end{equation}
let $z = (0 \ldots 0)^t \in K^{n-1}$, and for $i\leq 1\leq n-1$ consider the vector
\begin{equation*}
    u_i =
    \begin{pmatrix}
        z \\ z \\ \vdots \\ z \\ \alpha \\ z \\ \vdots \\ z
    \end{pmatrix}
    = \begin{pmatrix}
        0 \\ \vdots \\ 0 \\x_{1,1}\\ x_{1,2}\\ \vdots \\ x_{1,n-1}\\ 0 \\ \vdots \\ 0 
    \end{pmatrix}\in K^{n(n-1)}
\end{equation*}
where $\alpha$ is preceded by $i$-many $z$ blocks. Next let $A_1 = \alpha$ and for $2\leq i\leq n$ take
\begin{equation*}
    A_{n-1,i}:=A_i =
    -\begin{pmatrix}
        x_{i,1}\\ x_{i,2}\\ \vdots \\ x_{i,n-1}
    \end{pmatrix} \text{ and }
    A =
    \begin{pmatrix}
        A_1\\ A_2 \\ \vdots \\ A_n.
    \end{pmatrix}
\end{equation*} We claim that $\mathcal{H} \cdot M_n = 0$, where $M_n$ is the following block matrix
\begin{equation*}
    M_n:= (u_1 \text{ } u_2 \text{ } \ldots  \text{ } u_{n-1} \text{ } A)\in K^{n(n-1)\times n}.
\end{equation*}

In other words, we claim that
\begin{equation}\label{star-eq1}
    0 = \begin{pmatrix}
            B_2 \text{ }& -B_1 \text{ }&  \text{ }&  \text{ }&  \text{ }&  \text{ }& \\
            B_3  && - B_1\\
            B_4  &&& -B_1\\
            B_5 &&&&- B_1\\
            \vdots &&&&&\ddots\\
            B_n &&&&&& -B_1
        \end{pmatrix}
        \begin{pmatrix}
            \text{ }&\text{ }&\text{ }& \text{ }& \text{ }&A_1\\
            A_1 &&&&&A_2\\
            & A_1 &&&& A_3\\
            && A_1 &&& A_4\\
            &&&\ddots &A_1 & A_n
        \end{pmatrix}.
\end{equation}
Expanding out the previous expression via multiplication of the matrices block-wise, we see that (\ref{star-eq1}) is equivalent to both of the following statements holding
\begin{enumerate}
    \item[(i)] $-B_1A_1 = 0$;
    \item[(ii)]for $2\leq i\leq n$, $B_{i}A_1 - B_1A_i=0$.
\end{enumerate}
We easily verify (i) and (ii) using the definitions of the $A_i$ and $B_i$ matrices. Hence (\ref{star-eq1}) holds. Now (\ref{star-eq1}) combined with the fact that $\mathcal{H}$ has $|G||H| = n(n-1)$ many columns gives that
\begin{equation}\label{star-eq2}
    \rk (\mathcal{H})\leq n(n-1) - \rk (M_n).
\end{equation}
Our next step in proving (a) is to show that $M_n$ has rank $n$. $\rk (M_n)\leq n$ since $M_n$ has $n$ rows. Let $\phi: K\to K$ be the ring map induced by the prescription that $x_{i,j}\mapsto 0$ for $2\leq i\leq n$ and $1\leq j\leq n-1$, and $x_{1,j}\mapsto x_{1,j}$ for $1\leq j\leq n-1$. Let $\phi(M_n)$ be the matrix obtained by applying $\phi$ to each entry of $M_n$. After a column permutation, $\phi(M_n)$ is the matrix $\text{diag}(A_1,A_1,\ldots,A_1)\in K^{n(n-1)\times n}$ which clearly has linearly independent columns. Since $\phi$ is a substitution map, we now have
\begin{equation*}
    n \geq \rk (M_n) \geq \rk (\phi(M_n)) = n.
\end{equation*}
Then by (\ref{star-eq2})
\begin{equation*}
    \rk (\mathcal{H})\leq n(n-1) - n = (n-1)^2 -1 = e(St_n)e(C_{n-1})-1 = e(G)e(H) - 1
\end{equation*}
so that (a) is proven.\\

Now we prove (b). So assume that $n\geq 5$ is odd. For $2 \leq i\leq n$ take
\begin{equation*}
    C_i =
    \begin{pmatrix}
        -x_{i,1}\\ 0 \\ -x_{i,3} \\ 0 \\ \vdots \\ -x_{i,n-4} \\ 0 \\ -x_{i,n-2}\\ 0
    \end{pmatrix},
    C_1 =
    \begin{pmatrix}
         0 \\ x_{1,2}\\ 0 \\ x_{1,4}\\ \vdots \\ 0 \\ x_{1,n-3} \\ 0 \\ x_{1,n-1}
    \end{pmatrix},
    D_i =
    \begin{pmatrix}
         0 \\ -x_{i,2}\\ 0 \\ -x_{i,4}\\ \vdots \\ 0 \\ -x_{i,n-3} \\ 0 \\ -x_{i,n-1}
    \end{pmatrix},
    D_1 =
    \begin{pmatrix}
        x_{1,1}\\ 0 \\ x_{1,3} \\ 0 \\ \vdots \\ x_{1,n-4} \\ 0 \\ x_{1,n-2}\\ 0
    \end{pmatrix},
\end{equation*}
and consider the matrices
\begin{equation*}
    C:= \begin{pmatrix}
        C_1\\ C_2\\ \vdots \\ C_n
    \end{pmatrix},
    D:= \begin{pmatrix}
        D_1\\ D_2\\ \vdots \\ D_n
    \end{pmatrix}.
\end{equation*}
We claim that $\mathcal{H} L_n = 0$, where $L_n = (u_1 \ldots u_{n-1} \text{ }C \text{ } D)\in K^{n(n-1)\times n+1}$. By our statement (i) from the proof of (a), we already have $\mathcal{H}u_i = 0$ for $1\leq i\leq n-1$. So it remains to show that $\mathcal{H}C = \mathcal{H}D = 0$. As in, we must prove that
\begin{equation}\label{star-eq3}
    0 = 
    \begin{pmatrix}
            B_2 \text{ }& -B_1 \text{ }&  \text{ }&  \text{ }&  \text{ }&  \text{ }& \\
            B_3  && - B_1\\
            B_4  &&& -B_1\\
            B_5 &&&&- B_1\\
            \vdots &&&&&\ddots\\
            B_n &&&&&& -B_1
        \end{pmatrix}
        \begin{pmatrix}
            C_1  \text{ }& D_2\\
            C_2 & D_2\\
            C_3 & D_3\\
            C_4 & D_4 \\
            \vdots & \vdots\\
            C_n & D_n
        \end{pmatrix}.
\end{equation}Our claim (\ref{star-eq3}) is equivalent to the following two statements both holding
\begin{enumerate}
    \item[(iii)] $B_iC_1 - B_1C_i = 0$ for $2\leq i\leq n$;
    \item[(iv)] $B_i D_1 - B_1D_i = 0$ for $2\leq i\leq n$,
\end{enumerate}
which can be easily seen using the definition of the matrices $B_i,C_i$ and $D_i$. Thus, (\ref{star-eq3}) holds, so that $\rk (\mathcal{H}) \leq n(n-1) - \rk (L_n)$. Now to finish the proof of (b), it remains to show that $L_n$ has rank $n+1$. Firstly, $\rk (L_n) \leq n+1$. On the other hand, after a permutation of the columns and rows of $\phi(L_n)$, we obtain the following rank $n+1$ matrix
\begin{equation*}
    \begin{split}
        &\text{diag}(U,V,\alpha,\alpha,\ldots,\alpha)\in K^{n(n-1)\times n+1} ,
    \\&\text{ where }
    U:=
    \begin{pmatrix}
        x_{1,1}\\ x_{1,3}\\ \vdots \\ x_{1,n-2}
    \end{pmatrix} \in K^{(n-1)/2}
    \text{, and }V:=
    \begin{pmatrix}
        x_{1,2}\\ x_{1,4}\\ \vdots \\ x_{1,n-1}
    \end{pmatrix}\in K^{(n-1)/2}.
    \end{split}
\end{equation*}
Consequently,
\begin{equation*}
    \begin{split}
        n+1 &\geq \rk (L_n) \geq \rk (\phi(L_n)) = \rk (\text{diag}(U,V,\alpha,\alpha,\ldots,\alpha))
    \\&= \rk (U )+ \rk (V )+ (n-1)\rk(\alpha) = n+1,
    \end{split}
\end{equation*}
which completes the proof.
\end{proof}

Now Theorem \ref{thmA-2} (1) follows from Proposition \ref{propB} (stating $\ell(J_{G,H})=\rk(\mathcal{H}(G,H)$) together with
the above Lemma~\ref{lem-star} and Lemma \ref{lem-subgr} (2) (monotonicity). $\square$

We end this section by proving a corollary of Theorem \ref{thmA-2} which states in essence that Proposition \ref{prop-leaf3} cannot be strengthened.

\begin{corollary}
    For every cycle $G:=C_n$ and every graph $G'$ there exists a graph $H$ containing $G'$ and a graph $H'$ obtained by adding a leaf to $H$ such that $\ell(J_{G,H'}) < \ell(J_{G,H}) + \ell(J_G)$.
\end{corollary}
\begin{proof} We start with showing that the claimed inequality holds for some star graph $H$ (which may not contain $G'$). 
    Suppose not. Note how $St_{n+1}$ is obtained by adding $n-2$ leaves to $P_3$, while Theorem \ref{thmA} (2) implies that $\ell(J_{C_n,P_3}) = e(C_n)e(P_3)$. This combined with inductively applying our assumption $n-2$ times that $\ell(J_{G,H'}) = \ell(J_{G,H}) + \ell(J_G)$ for any graph $H$ and any graph $H'$ obtained by adding a leaf to $H$, yields that $\ell(J_{C_n,St_{n+1}}) = e(C_n)e(P_3) + (n-2)e(C_n) = e(C_n)e(St_{n+1})$, which contradicts Theorem \ref{thmA-2} (2). 
    
    Now let $H$ be a star graph such that $\ell(J_{G,H'}) < \ell(J_{G,H}) + \ell(J_G)$ for some graph $H'$ obtained by adding a leaf to $H$. Let $H_1$ be the disjoint union of $H$ and $G'$, and let $H_1'$ be the disjoint union of $H'$ and $G'$. Then by Remark~\ref{rmk-comp} we have $\ell(J_{G,H_1'})< \ell(J_{G,H_1})+\ell(J_G)$, as desired.
\end{proof}

\section{Proof of Theorem \ref{thmA-2} (2)}\label{sec-9:pf2}

In this section we prove Theorem \ref{thmA-2} (2). We shall maintain some of the notation from Section \ref{sec-8:prfA_part2}. In particular fix positive integers $m\geq 4$ and $n\geq 3$ and consider the matrices $B_{r,i}$ of Section \ref{sec-8:prfA_part2}. For this section, take $B_i:= B_{n,i}$ and
\begin{equation*}
U_i =
\begin{pmatrix}
        x_{i,1}\\ x_{i,2} \\ \vdots \\ x_{i,n}
    \end{pmatrix},
    A_i:=
    (-1)^{i+1}U_i = 
    (-1)^{i+1}
    \begin{pmatrix}
        x_{i,1}\\ x_{i,2} \\ \vdots \\ x_{i,n}
    \end{pmatrix},
    L:= 
    \begin{pmatrix}
        A_1 \\ A_2 \\ \vdots \\ A_m
    \end{pmatrix}\in K^{mn\times 1}.
\end{equation*}

\begin{proof}[Proof of Theorem \ref{thmA-2} (2)]

We proceed initially with no assumption on the parity of $m$ or $n$. Let $\mathcal{H} = \mathcal{H}(C_m,C_n)$. Using (\ref{hgh-df}) we compute that
\begin{equation*}
    \mathcal{H}
    =
    \begin{pmatrix}
        B_2 \text{ }&-B_1 \text{ }&\text{ }&\text{ }&\text{ }&\text{ }&\text{ }&\\
        & B_3 & -B_2 &&&&&\\
        &&B_4 & -B_3&&&&&\\
        &&&B_5&-B_4 &&&&\\
        &&&&&\ddots &&\\
        &&&&&&B_m & -B_{m-1}\\
        B_m &&&&&&& -B_1
    \end{pmatrix}\in K^{mn\times mn}.
\end{equation*}

By Lemma~\ref{lem-switch} and Lemma \ref{lem-subgr} (2), it suffices to show that whenever $m$ is even, we have
\begin{equation}\label{cycle-eq1}
    0 = \mathcal{H}L = 
    \begin{pmatrix}
        B_2 \text{ }&-B_1 \text{ }&\text{ }&\text{ }&\text{ }&\text{ }&\text{ }&\\
        & B_3 & -B_2 &&&&&\\
        &&B_4 & -B_3&&&&&\\
        &&&B_5&-B_4 &&&&\\
        &&&&&\ddots &&\\
        &&&&&&B_m & -B_{m-1}\\
        B_m &&&&&&& -B_1
    \end{pmatrix}
    \begin{pmatrix}
        A_1 \\ A_2 \\ \vdots \\ A_m
    \end{pmatrix}.
\end{equation}

Via block multiplication we see that (\ref{cycle-eq1}) holding when $m$ is even follows from the conjunction of the following two statements:

\begin{enumerate}
    \item[(i)] $B_{i+1}A_i = B_iA_{i+1}$ for $1\leq i\leq m-1$, and 
    \item[(ii)] $B_mU_1 = (-1)^m B_1U_m$.
\end{enumerate}
We see (i) and (ii) easily by computing this matrix multiplication using the definitions of the matrices $U_i$, $B_i$ and $A_i$ (obtaining (ii) is where we used that $m$ is even). This completes the proof.
\end{proof}
In view of Conjecture~\ref{question-cycle}, note that in fact $\mathcal{H}(C_m,C_n)\cdot L = 0$ if and only if $m$ is even, since (ii) fails when $m$ is odd.

\section{Proof of Theorem \ref{thmA-2} (3)}\label{sec10:pf-2_pt2}

In this section we prove Theorem \ref{thmA-2} (3) which states that a pair of graphs $(G,H)$ is dependent if one of $G$ or $H$ contains a cycle and the other contains a cycle with a leaf; over suitable field characteristics.
By Lemma~\ref{lem-subgr}, 
it follows from the following lemma.

\begin{lemma}\label{lem-cycle_leaf}
Fix positive integers $m,n\geq 3$. 
    Let $G = C_n$ and let $H$ be the graph on $m+1$ vertices obtained by appending the leaf $\{1,m+1\}$ to the cycle $C_m$. Then $\rk (\mathcal{H}(G,H))<e(G)e(H)$.
\end{lemma}
\begin{proof}
Let $\mathcal{H} = \mathcal{H}(G,H)$. Since $\mathcal{H}$ is square of size $e(G)e(H)$, it suffices to find a vector $L\in K^{n(m+1)}$ for which $\mathcal{H}L = 0$. Denote $z = (0\ldots 0)^t\in K^m$. For $1\leq i\leq n$, let
\begin{equation*}
    A_i =
    \begin{pmatrix}
        z \\ x_{i,1}
    \end{pmatrix},
    \text{ and }
    L =
    \begin{pmatrix}
        A_1 \\ A_2 \\ \vdots \\ A_n
    \end{pmatrix}
    =
    \begin{pmatrix}
        0\\ \vdots \\ 0\\ x_{1,1}\\ \vdots \\ 0 \\ \vdots \\ 0\\ x_{n,1}
    \end{pmatrix}.
\end{equation*}
For $1\leq i\leq n$ take
\begin{equation*}
    E_i =
    \begin{pmatrix}
        x_{i,2} &\text{ }-x_{i,1}  &\text{ } &\text{ } &\text{ } &\text{ } &\text{ }\\
        & x_{i,3}& -x_{i,2}\\
        &&& \ddots\\
        &&&&x_{i,m} & -x_{i,m-1}\\
        x_{i,m}&&&&&-x_{i,1}\\
        x_{i,m+1}&&&&&&-x_{i,1}
    \end{pmatrix}.
\end{equation*}
By (\ref{hgh-df}) we readily find that
\begin{equation*}
    \mathcal{H}
    =
    \begin{pmatrix}
        E_2 &\text{ }-E_1 &\text{ } &\text{ } &\text{ } &\text{ }\\
        & E_3 &-E_2\\
        &&& \ddots\\
        &&&& E_n &-E_{n-1}\\
        E_n&&&&&-E_1
    \end{pmatrix}.
\end{equation*}
We claim that $\mathcal{H}L = 0$, which is equivalent to saying that
\begin{equation}\label{eq-cycle_leaf}
    0 =
    \begin{pmatrix}
        E_2 &\text{ }-E_1 &\text{ } &\text{ } &\text{ } &\text{ }\\
        & E_3 &-E_2\\
        &&& \ddots\\
        &&&& E_n &-E_{n-1}\\
        E_n&&&&&-E_1
    \end{pmatrix}
    \begin{pmatrix}
        A_1 \\ A_2 \\ \vdots \\ A_n
    \end{pmatrix}.
\end{equation}
By block multiplication, we see that (\ref{eq-cycle_leaf}) is equivalent to saying that both of the following statements hold:
\begin{enumerate}
    \item[(i)] $0 = E_{i+1}A_i-E_iA_{i+1}$ for $1\leq i\leq n-1$, and 
    \item[(ii)] $0 = E_nA_1 - E_1A_n$.
\end{enumerate}
On the other hand, the definitions of the $E_i$ and $A_j$ provide that $E_iA_j = E_jA_i$ for $1\leq i,j\leq n$, so (i) and (ii) follow. Hence $\mathcal{H}L = 0$ and the lemma is proven.
\end{proof}

Indeed, Theorem \ref{thmA-2} (3) now follows.

\begin{proof}[proof of Theorem \ref{thmA-2} (3)]
    By Lemma \ref{lem-switch} we can assume that $G$ contains a cycle and $H$ contains a cycle with a leaf. Then due to Lemma \ref{lem-subgr} (2), it is enough to prove the case when $G$ is a cycle and $H$ is a cycle with a leaf. By relabeling the vertices of $H$ we can assume the leaf $e\in E(H)$ in $H$ is attached to the vertex $1\in V(H)$, i.e. that $e = \{1,m+1\}$. Now our claim is exactly the statement of Lemma \ref{lem-cycle_leaf}.
\end{proof}


\section*{Acknowledgments}

Both authors would like to thank the Center for Mathematical Sciences and Applications at Harvard, and the first author also thanks the Hebrew University of Jerusalem, for their support during this project.
Both authors are partially supported by the Israel Science Foundation grant ISF-687/24.

\textbf{AI use statement.}
The authors used Gemini Pro during the revision stage of this paper for identifying typos and to double check calculations. All original results, proofs, and mathematical content predate the use of AI assistance. All AI-suggested edits were reviewed and approved by the authors, who bear sole responsibility for the content.

\bibliographystyle{amsplain}

\bibliography{biblio}

\end{document}